\documentclass[english]{amsart}

\usepackage{enumerate} 
\usepackage{amssymb}
\usepackage{array}
\usepackage{mathdots}
\usepackage{tkz-fct} 
\usepackage{mathtools}
\usepackage[curve,all]{xy}
\usepackage{tikz-cd}
\usepackage{tabu}
\usepackage{multirow}
\xyoption{matrix}
\xyoption{curve}
\usepackage{bm}
\xyoption{color}
\xyoption{line}
\xyoption{arc}

\usepackage[shortlabels]{enumitem}
\setlist[enumerate]{label=\rm{(\arabic*)}}
\setlist[enumerate,2]{label=\rm({\it\roman*})}
\setlist[itemize]{label=\raisebox{0.25ex}{\tiny$\bullet$}}
\usepackage[backref=false, colorlinks, linktocpage, citecolor = blue, linkcolor = blue]{hyperref}

\newcommand{\C}{\mathbb{C}}
\newcommand{\Z}{\mathbb{Z}}
\newcommand{\R}{\mathbb{R}}
\newcommand{\F}{\mathbb{F}}
\newcommand{\p}{\mathbb{P}}
\newcommand{\w}{\omega}
\newcommand{\V}{\mathcal{V}}
\newcommand{\RR}{\mathcal{R}}
\newcommand{\D}{\mathcal{D}}
\newcommand{\W}{\mathcal{W}}
\newcommand{\Spec}{\mathrm{Spec}}
\renewcommand{\k}{\mathrm{k}}
\renewcommand{\O}{\mathcal{O}}

\DeclareMathOperator{\SL}{SL}
\DeclareMathOperator{\SU}{SU}
\DeclareMathOperator{\Spin}{Spin}
\DeclareMathOperator{\SO}{SO}
\DeclareMathOperator{\PGL}{PGL}
\DeclareMathOperator{\Gal}{Gal}

\DeclareMathOperator{\et}{ and }
\DeclareMathOperator{\Aut}{Aut}

\usepackage[toc,page]{appendix}
\theoremstyle{definition}

\theoremstyle{definition}
\theoremstyle{plain}
\newtheorem{theorem}{Theorem}[section]
\newtheorem*{theoremaux}{Main Theorem}
\newtheorem{proposition}[theorem]{Proposition}
\newtheorem{lemma}[theorem]{Lemma}
\newtheorem{corollary}[theorem]{Corollary}
\theoremstyle{definition}
\newtheorem{definition}[theorem]{Definition}
\newtheorem*{definition*}{Definition}

\newtheorem{example}[theorem]{Example}
\theoremstyle{remark}
\newtheorem{remark}[theorem]{Remark}
\newtheorem*{remark*}{Remark}

\title{Real forms of minimal $\SL_2$-threefolds}
\author[Lucas Moulin]{Lucas Moulin}
\thanks{The IMB receives support from  the EIPHI Graduate School (contract ANR-17-EURE-0002).}
\address{Institut de Math\'{e}matiques de Bourgogne, UMR 5584 CNRS, Universit\'{e} de Bour-gogne, F-21000 Dijon, France}
\email{lucas.moulin@u-bourgogne.fr}

\begin{document} 

\maketitle
\begin{abstract}
We complete the classification of the real forms of almost homogeneous $\SL_2$-threefolds. More precisely, we use the Luna-Vust theory to determine the real forms of minimal smooth complete $\SL_2$-varieties containing an orbit isomorphic to $\SL_2/H$, where $H$ is a finite cyclic subgroup of $\SL_2$. Moreover, we study the rationality and the set of real points of those varieties.
\end{abstract}
\tableofcontents

\section{Introduction} \label{section: Intro}

\subsection{Setting} \label{section: setting}

Let $X$ be a complex algebraic variety. A classical problem in algebraic geometry is to determine the real forms of $X$, i.e. the real algebraic varieties $W$ such that $W_\C := W \times_{\Spec(\R)} \Spec(\C) \simeq X$. This problem is usually divided in two parts: whether $X$ admits a real form or not, and if so, determine how many there are (up to isomorphism).

In the present article, we work on a slightly different problem. We are interested in algebraic varieties endowed with an algebraic group action, and thus we only consider real forms which are compatible with this algebraic group action. More precisely, let $G$ be a complex reductive algebraic group and let $F$ be a real algebraic group such that $G \simeq F_\C$ as complex algebraic groups. Now let $X$ be a complex algebraic $G$-variety. We say that $W$ is a \textit{real F-form} of $X$ if $W_\C \simeq X$ as complex algebraic $G$-varieties. When the real algebraic group $F$ is clear from the context, we will just refer to real $F$-forms as real forms.

An important invariant of $G$-varieties is the complexity. For $B \subset G$ a Borel subgroup, the \textit{complexity} of a complex algebraic $G$-variety is defined as the maximal codimension of a $B$-orbit. Note that this does not depend on the choice of $B$. Forms over any perfect base field of complexity-zero varieties (more commonly known as spherical varieties) have already been classified by Huruguen \cite{Hur11}, Wedhorn \cite{Wed18}, and Borovoi-Gagliardi \cite{BG21}. For $T$ a torus, forms of complexity-one $T$-varieties have also been studied by Langlois in \cite{Lan15} and by Gillard in \cite{Gil22}.

Another invariant of a homogeneous space $G/H$ is the rank, which we denote rk$(G/H)$. Suppose that rk$(G/H) \leq 1$. Then, Panyushev proved in \cite{Pan95} that either $G/H$ is a spherical homogeneous space or $G/H$ is obtained from a homogeneous $\SL_2$-threefold by parabolic induction. In this sense, almost homogeneous $\SL_2$-threefolds are therefore the easiest case of complexity-one almost homogeneous varieties. Thus, let us now suppose that $G=\SL_2$ and that $X$ is a complex almost homogeneous $\SL_2$-threefold, i.e. an $\SL_2$-variety with a dense open orbit $X_0 = \SL_2/H$, where $H$ is a finite subgroup of $\SL_2$. Note that, since a Borel subgroup of $\SL_2$ is of codimension 1, these varieties are of complexity one. In this setting, the Luna-Vust theory provides a combinatorial description of all the $G$-equivariant embeddings $X_0 \hookrightarrow X$ for a given $X_0$, up to $G$-equivariant isomorphisms (see \cite[\S~9]{LV83} for details or \S \ref{Luna-Vust} for an overview). 

More specifically, we will study the real forms of minimal smooth completions of $G/H$. Here we call a \textit{minimal smooth completion} of $G/H$ a smooth complete almost homogeneous $\SL_2$-threefold, with open orbit $G/H$, such that any $\SL_2$-equivariant birational morphism $X \rightarrow X'$, with $X'$ smooth, is an isomorphism. To simplify, we will also refer to a minimal smooth completion as a \textit{model}. Those varieties are well-known, they were first studied by Umemura in \cite{Ume88} as the connected component of their automorphism groups are maximal connected algebraic subgroups of the Cremona group Bir$(\p^3)$, and the projective ones were then classified by Nakano in \cite{Nak89} as an application of Mori theory. They arise naturally in many ways in algebraic geometry, and the list includes Fano varieties such as $\p^3$, the smooth quadric $Q_3$ and the degree five del Pezzo surface $Y_5$, as well as Umemura $\p^1$-bundles and the projectivization of Schwarzenberger bundles (see \cite{Ume88} and \cite{BFT23}). From the Luna-Vust theory point of view, the full classification and the description of the corresponding combinatorial data were obtained by Moser-Jauslin \cite{MJ87} when $H \subset \{\pm Id\}$ and by Bousquet \cite{Bou00} in the general case. 

Using a generalization of the Luna-Vust theory over arbitrary perfect fields, Moser-Jauslin and Terpereau determined in \cite{MJT20} the real forms of the minimal smooth completions of $\SL_2/H$ when $H$ is a non-cyclic finite subgroup. In their case, there is always a unique smooth completion which has one or two real $F$-forms for a fixed $F$. Here we will use a similar approach to determine the real forms of $\SL_2/H$ when $H$ is  a cyclic subgroup of order $n$. However, in this case, a few new difficulties arise compared to the non-cyclic case since the automorphism group $\Aut^G(X)$ can be infinite and there are many more cases to consider. Also the underlying complete $G$-varieties are not all projective, so we have to investigate more carefully the effectiveness of the Galois descent.

For a cyclic subgroup $H \subset \SL_2$ of cardinal $k$, $H$ is conjugate to the group $A_k$ generated by $\omega_k = \begin{bmatrix}
    \zeta_k & 0 \\ 0 & \zeta_k^{-1}
\end{bmatrix}$, where $\zeta_k$ is a primitive $k$-th root of unity. A diagrammatic representation of the combinatorial data of all the models of $\SL_2/H$, as well as the correspondence with the description of Nakano, can be found in the \hyperref[appendix: A]{Appendix}. The list of the projective ones given by Nakano (cf. \cite{Nak89}) is the following: \begin{itemize}[leftmargin=7mm]

\item[1)] \textbf{The Umemura $\p^1$-bundles $X_k(n,m) \rightarrow \F_k,$ $k\geq 1,$ $n,m \geq 0$}. We note $\F_k$ the Hirzebruch surface defined by $\F_k := \p(\O_{\p^1} \oplus \O_{\p^1}(-k))$, $p_k: \F_k \rightarrow \p^1$ the projection, and $C_\infty$ the negative section of $p_k$ defined by the projection $\O_{\p^1} \oplus \O_{\p^1}(-k) \rightarrow \O_{\p^1}(-k)$. There is a natural $\SL_2$-action on $\F_k$ and $C_\infty$ is a 1-dimensional closed orbit. There is a unique homogeneous non-trivial extension (cf. \cite{Ume85})
\[ 0 \rightarrow \O_{\F_k} \rightarrow E_k(n,m) \rightarrow p_k^*(\O_{\p^1}(2-nk)) \otimes \O_{\F_k}(-mC_{\infty}) \rightarrow 0.\]
We get a $\p^1$-bundle $\pi_k(n,m): X_k(n,m) \rightarrow \F_k$ by projectivizing the vector bundle $E_k(n,m) \rightarrow \F_k$. The variety $X_k(n,m)$ is a smooth completion of $\SL_2/A_k$, and it is minimal if and only if $m \neq 1$ and ($n \neq m$ if $k=1$).

\item[2)]  \textbf{The $\p^1$-bundles $W(k) \rightarrow \p^2$, $k\geq0$}. As we just saw, $X_1(k,k)$ is not minimal and we can equivariantly contract $\pi_1(k,k)^{-1}(C_\infty)$ to $\p^1$ to get a new smooth completion of $\SL_2/A_k$, which we call $W(k)$. The variety $W(k)$ is minimal if and only if $k \neq 2$.

\item[3)] \textbf{The decomposable $\p^1$-bundles $Y(a,b)\rightarrow \p^1 \times \p^1, a > b, a+b \leq 0$}. We set $Y(a,b) := \p(\O_{\p^1 \times \p^1} \oplus \O_{\p^1 \times \p^1}(a,b))$ where $O_{\p^1 \times \p^1}(a,b)$ is the line bundle over $\p^1 \times \p^1$ of bidegree $(a,b)$. We can define an $\SL_2$-action on $Y(a,b)$ such that the projection $\pi(a,b): Y(a,b) \rightarrow \p^1 \times \p^1$ is equivariant (we consider the diagonal action on $\p^1 \times \p^1$). The variety $Y(a,b)$ is an almost homogeneous threefold of $\SL_2/A_{-a-b}$. It is minimal if and only if $a \neq \pm 1$ and $b \neq \pm 1$.

\item[4)] \textbf{The pull-backs of Schwarzenberger $\p^1$-bundles $Y_0(k)$, $k = 2m > 0$}. Given $k = 2m > 0$, define $a=1+m$ and $b=1-m$. Then we can show that Ext$^1(\O(a,b),\O)$ has a unique $\SL_2$-invariant 1-dimensional subspace $V$, which determines an extension
\[
    0 \rightarrow \O \rightarrow G_0(k) \rightarrow M \rightarrow 0,
\]
where $G_0(k)$ is a homogeneous vector bundle of rank 2 and $M:= \O_{\p^1 \times \p^1}(a,b)$. $G_0(k)$ has an $\SL_2$-linearization such that the above exact sequence is $\SL_2$-exact. We define $Y_0(k) := \p(G_0(k))$. There is an induced $\SL_2$-action on $Y_0(k)$ such that the $\p^1$-bundle projection $\pi_0(k): Y_0(k) \rightarrow \p^1 \times \p^1$ is equivariant. $Y_0(k)$ has an open orbit equivariantly isomorphic to $\SL_2/A_k$. It is minimal if and only if $k \neq 4$.

\item[5)] \textbf{The decomposable $\p^2$-bundles $\pi(r): Z(r) \rightarrow \p^1$, $r \neq 1$}. Define $E(r):= \O_{\p^1}^{\oplus2} \oplus\O_{\p^1}(r)$ and $Z(r):= \p(E(r))$. We can define an $SL_2$-action on $Z(r)$ such that the projection $\pi(r)$ is equivariant.  $Z(r)$ is minimal if and only if $r \neq 1$. For $r \neq 1$, it has an open orbit isomorphic to $\SL_2/A_{|r-1|}$.

\item [6)] \textbf{$\p^2 \times \p^1 \simeq \p(R_2) \times \p(R_1)$}. We note $R_i$ the unique irreducible representation of $\SL_2$ of dimension $i+1$. Then $\p(R_2) \times \p(R_1)$ is naturally equipped with a diagonal action of $\SL_2$. It is a minimal smooth completion of $\SL_2/A_4$.

\item[7)] Finally we have two more isolated models, namely $\p^3 \simeq \p(R_1 \oplus R_1)$ and $Q_3 \subset \p^4 \simeq \p(R_1 \oplus R_1 \oplus R_0)$. The group $\PGL_2$ identifies with an open subset of $\p^3$. The natural $\SL_2$-action on $\PGL_2$ extends to the one on $\p^3 \simeq \p(R_1 \oplus R_1)$ so we get a smooth completion of $\PGL_2$. In a similar way, $\SL_2$ is defined by the equation $ad-bc=1$ in $\mathbb{A}^4$ so we set $Q_3 := \{ad-bc=e^2\} \subset \p^4$. It is easy to show that the action of $\SL_2$ on itself by translation extends to $Q_3$. Furthermore, $Q_3$ and $\p^3$ are both minimal since they have Picard number equal to 1.

\end{itemize}
\subsection{Our results}
In this paper, we determine the equivariant real forms of all these models using the Luna-Vust theory. The group $\SL_2$ has two real forms up to isomorphism, namely $\SL_{2,\R}$ and $\Spin_{3,\R}$. Their set of real points are $\SL_2(\R)$ and $\SU_2(\C)$ respectively. Therefore, we classify the equivariant real forms for these two groups. Moreover we determine which real forms (in the classical sense) appear as an equivariant real form. Finally, we also determine which equivariant real forms are rational and whether they have real points. Recall that a rational variety always has real points. In this case we show that the converse is true. The main result of this article is the following (see Section \ref{section: theorem} for the proof):

\begin{theoremaux} \label{Main Theorem}
    Let $G=\SL_{2,\C}$, $H \subset G$ a finite cyclic subgroup, and $X$ a complex projective minimal smooth completion of $G/H$. Then the equivariant real forms of $X$ are either rational, or they have no real points. Furthermore the number of rational real forms and real forms with an empty real locus, up to isomorphism, are listed in the table below for each model $X$. Each column indicates the following number : \begin{enumerate}[leftmargin=7mm, label=(\roman*)]
        \item Number of real forms.
        \item Number of $\SL_{2,\R}$-real forms.
        \item Number of $\Spin_{3,\R}$-real forms.
        \item Number of real forms that appear in columns (ii) or (iii).
        
    \end{enumerate}
   
\end{theoremaux}

\[\tabulinesep=0.8mm
\hspace{-5mm}\begin{tabu}{|c|c|c|c|c||c|c|c|c|}
	\hline 
	\multirow{2}{*}{\text{Models}} &   \multicolumn{4}{c||}{\text{Rational real forms}} &
	  \multicolumn{4}{c|}{\text{Real forms w/o real points}} \\ \cline{2-9}
    & (i) & (ii) & (iii) & (iv) & \hspace*{1mm}(i) & \hspace*{1mm}(ii) & \hspace*{1mm} (iii) & \hspace*{1mm}(iv)\\
    \hline
	Z(1-k),X_k(n,m),\text{ $k$ even} &1&1&0&1&1&0&1&1\\
	\hline
	Z(1-k),X_k(n,m), \text{ $k$ odd} &1&1&0&1&0&0&0&0\\
	\hline
        Y(a,-a) &2&2&1&2&3&2&1&3\\
	\hline
	Y(a,b), \text{ $a+b \neq 0$ even}  &1&1&0&1&2&0&1&1\\
	\hline
	Y(a,b), \text{ $a+b$ odd}  &1&1&0&1&2&0&0&0\\
	\hline
    Y_0(2) &2&2&0&2&4&2&0&2 \\
    \hline
    Y_0(k), &2&1&1&2&2&1&1&2 \\
    \hline
    W(1) &2&2&1&2&1&0&1&1\\
    \hline
    W(n+1), \text{ $n \geq 2$} &1&1&0&1&0&0&0&0\\
    \hline
    Q_3 \subset \p(R_1 \oplus R_1 \oplus R_0)&2&1&1&2&1&1&1&1\\
    \hline
    \p^3 \simeq \p(R_1 \oplus R_1) &1&1&1&1&1&1&1&1 \\
    \hline
    \p(R_2) \times \p(R_1) &1&1&1&1&1&0&0&0\\
    \hline
\end{tabu} \]

\begin{remark*}
    Note that for $\p(R_2) \times \p(R_1) \simeq \p^2 \times \p^1$ there are two real forms in total, but only the rational one is equivariant under an action of a real form of $\SL_2$. Actually, we see that the rational real form is equivariant under an action of both $\SL_{2,\R}$ and $\Spin_{3,\R}$.
\end{remark*}

\begin{remark*} We also determine the equivariant real forms of the nonprojective minimal smooth completions of $G/H$, however we do not study the rationality nor the real points of those varieties (see Section \ref{section: forms}).
\end{remark*}

\subsection{Structure of the article} Section \ref{section: preliminaries} contains preliminary results necessary for the rest of the paper. In Section \ref{section: real structures}, we review real group structures on reductive algebraic groups as well as real structures on varieties endowed with an algebraic group action. We also explain how they correspond to real forms. Then, in Section \ref{Luna-Vust} we see how Luna-Vust theory specializes to classify the equivariant embeddings of homogeneous $\SL_2$-threefolds, and how a generalization of this theory can be used to classify real forms. In Section \ref{section: group}, we explain how to compute the equivariant automorphism groups of equivariant embeddings of homogeneous spaces using the combinatorial data of the Luna-Vust theory. We then compute those groups for all the minimal smooth completions of $\SL_2/H$, where $H$ is a cyclic subgroup. In Section \ref{section: forms}, we determine all the real forms of the $\SL_2$-threefolds listed previously. Finally, Section \ref{section: theorem}, we study the rationality and the real locus of the real forms we obtained previously, and then prove the \hyperref[Main Theorem]{Main Theorem}.

\subsection{Acknowledgements}

I would like to thank my advisors Ronan Terpereau and Lucy Moser-Jauslin for their comments and suggestions, which helped in writing and improving this article.

\section{Preliminaries} \label{section: preliminaries}

\subsection{Real structures and real forms}\label{section: real structures}

We start by recalling the notions of real forms and real structures on a given complex algebraic group $G$ and then on a $G$-variety, and how the notions are related. More details can be found in \cite{MJT21}.

\begin{definition}
    Let $G$ be a complex algebraic group.
    \begin{enumerate}[label=(\roman*)]
        \item A \emph{real form} of $G$ is a pair ($F$,$\Theta$) with $F$ a real algebraic group and $\Theta : G \rightarrow F_{\C} = F \times_{\Spec(\R)} \Spec(\C)$ an isomorphism of algebraic groups. (To simplify notation we will drop the isomorphism $\Theta$ and simply write that $F$ is a form of $G$).
        
        \item A \emph{real group structure} $\sigma$ on $G$ is a scheme involution on $G$ such that the diagram 
        \[
        \xymatrix@R=4mm@C=2cm{
         G \ar[rr]^{\sigma} \ar[d]  && G \ar[d] \\
          \Spec(\C)  \ar[rr]^{\gamma^*\ =\ \Spec(z \mapsto \overline{z})} && \Spec(\C)  
        }
        \]
        commutes and 
        \[
        \iota_G \circ  \sigma = \sigma \circ \iota_G \text{ and } m_G \circ (\sigma \times \sigma) = \sigma \circ m_G,
        \] where $\iota_G : G \rightarrow G$ is the inverse morphism and $m_G : G \times G \rightarrow G$ is the multiplication morphism.
        
        \item Two real group structures $\sigma$ and $\sigma'$ are equivalent if there exists a group automorphism $\varphi \in \Aut(G)$ such that $\sigma' = \varphi \circ \sigma \circ \varphi^{-1}$.
        
        \item If $G$ is a complex algebraic group with a real structure $\sigma$, then the set of fixed point $G(\C)^\sigma$ is called the \emph{real locus} of $(G,\sigma)$.
    \end{enumerate}
\end{definition}

There is a correspondence between real group structures on $G$ and real forms of $G$ given as follows (see \cite[\S~1.4]{FSS98}).

    $\bullet$ If $\sigma$ is a real group structure on $G$, one associates the real algebraic group $G_0 := G/\langle\sigma\rangle$ and the isomorphism $\Theta$ given by $(q,f)$, where $q: G \rightarrow G_0$ is the quotient morphism and $f: G \rightarrow \text{Spec($\C$)}$ is the structure morphism.

    $\bullet$ To a real form $F$ of $G$, one associates the real group structure given by $\sigma : G \simeq F_\C \xrightarrow{Id \times Spec(z\mapsto \bar{z})} F_\C \simeq G$.

    Moreover, two real forms of $G$ are isomorphic (as real algebraic groups) if and only if the corresponding real group structures are equivalent.

\begin{example}
For $G = \SL_2$, let $\sigma_s$ be the real group structure defined by $\sigma_s(g) = \overline{g}$ where $\overline{g}$ is the complex conjugate of $g$ and let $\sigma_c$ be the real group structure defined by $\sigma_c(g) = \prescript{t}{}\sigma_s(g)^{-1}$. Note that $\sigma_c(g) = e\sigma_s(g)e^{-1}$ for all $g \in G$, where $e = \begin{bmatrix}
    0 & 1 \\ -1 & 0
\end{bmatrix}$. The real loci of $\sigma_s$ and $\sigma_c$ are respectively $\SL_2(\R)$ and $\SU(2)$, and the corresponding real forms are respectively $\SL_{2,\R}$ and $\Spin_{3,\R}$. Also, any real group structure on $G$ is strongly equivalent to either $\sigma_s$ or $\sigma_c$.
\end{example}

We now define real forms and equivariant real structures for a $G$-variety.

\begin{definition}
    Let $(G,\sigma)$ be a complex algebraic group with a real structure, and let $X$ be a $G$-variety. We denote by $F = G/ \langle \sigma \rangle$ the corresponding real form of $G$.
    
    $\bullet$ An \emph{$(\R,F)$-form} of $X$ is an $F$-variety $Z$ together with an isomorphism $Z_\C  \simeq X$ of $G$-varieties, where $G$ acts on $Z_\C$ through $F_\C \simeq G$.
    
    $\bullet$ A \emph{$(G,\sigma)$-equivariant real structure} on $X$ is a scheme involution $\mu$ on $X$  such that the diagram 
\[
\xymatrix@R=4mm@C=2cm{
    X \ar[rr]^{\mu} \ar[d]  && X \ar[d] \\
    \Spec(\C)  \ar[rr]^{\gamma^*\ =\ \Spec(z \mapsto \overline{z})} && \Spec(\C)  
  }
  \]
commutes, and such that
\[ \forall g \in G(\C), \forall x \in X(\C),\ \mu(g\cdot x)=\sigma(g)\cdot \mu(x).
\]   

$\bullet$ Two $(G,\sigma)$-equivariant real structure $\mu_1$ and $\mu_2$ are called equivalent if there exists a $G$-equivariant automorphism $\varphi$ such that $\mu_2 = \varphi \circ \mu_1 \circ \varphi^{-1}$.

$\bullet$ A $(G,\sigma)$-equivariant real structure $\mu$ on $X$ is \emph{effective} if $X$ is covered by $\langle \mu \rangle$-stable affine open subsets. In this case, the categorical quotient $X/ \langle \mu \rangle$, which always exists as a real algebraic space, is in fact a real algebraic variety.

$\bullet$ The \emph{real locus} of $(G,\sigma)$-equivariant real structure $\mu$ on $X$ is the set of fixed points $X(\C)^\mu$; it identifies with the set of $\R$-points of the corresponding real form $X/ \langle \mu \rangle$. Moreover, it is endowed with an action of $G(\C)^\sigma$.
\end{definition}

As for algebraic groups, there is a bijection between isomorphism classes of real $F$-forms of $X$ (as real algebraic varieties) and equivalence classes of effective $(G,\sigma)$-equivariant real structures on $X$ (see \cite[\S~5]{Bor20})

Given the existence of a $(G,\sigma)$-equivariant real structure, one can use Galois cohomology to parametrize all of them. We note $\Gamma = \text{Gal}(\C/\R) = \{1,\gamma\}$. We define a $\Gamma$-group structure on $\Aut^G(X)$ as follows:
\[\text{inn}_{\mu}: \Gamma \times \Aut^G(X) \rightarrow \Aut^G(X), (\gamma,\varphi) \mapsto \mu_{\gamma} \circ \varphi \circ \mu_{\gamma}^{-1}.\]
The next result is then a straightforward consequence of the definition of 1-cocycles.

\begin{proposition} \emph{(\cite[Lemma 2.11]{MJT21}\label{prop:Z})}
	Let $(G,\sigma)$ be a complex algebraic group with a real group structure, and let $(X,\mu_0)$ be a homogeneous space with a $(G,\sigma)$-equivariant real structure.
	There is a bijection of pointed sets
	\[
	\begin{array}{ccl}
		 Z^1(\Gamma,\Aut_{\C}^{G}(X)) & \xrightarrow{\sim} & \left\{\text{$(G,\sigma)$-equivariant real structure on $X$}
		\right\} \\
	\varphi & \mapsto & \varphi \circ \mu_0 \\
	\end{array}
	\]
\end{proposition}

\begin{remark}
    In fact the correspondence above induces a bijection between the equivalence classes of $(G,\sigma)$-equivariant real structure on $X$ and $H^1(\Gamma,\Aut_{\C}^{G}(X))$. However, we will not use this result in this paper.
\end{remark}

\subsection{Luna-Vust Theory for equivariant embeddings of $\boldsymbol{\SL_2/H}$} \label{Luna-Vust}

We now recall how the Luna-Vust theory specializes to classify the equivariant embeddings of homogeneous $\SL_2$-threefolds. This section is inspired by \cite[Appendix B]{MJT20}. For more information on the Luna-Vust theory in the general case, see \cite{LV83} and \cite[\S~3]{Tim11}.

This description was given by Luna-Vust \cite[\S~9]{LV83} for equivariant embeddings of $\SL_2$, and by Moser-Jauslin \cite{MJ87} and Bousquet \cite{Bou00} for equivariant embeddings of $\SL_2/H$.

\begin{itemize}[leftmargin=2mm]
    \item \emph{The colors.} We start by describing the set of colors of $G/H$ where $H \subset G = \SL_2$ is a finite subgroup. We fix $B = \begin{bmatrix}
    * & 0 \\ * & *
\end{bmatrix}$ and we denote the coordinate ring of $G$ by $\C[G] = \C[w,x,y,z]/(xy-wz-1)$ with $g = \begin{bmatrix}
    x & w \\ z & y
\end{bmatrix} \in G$. In general, the set of colors is defined by
\[\D^B = \D^B(G/H) =  \{B\text{-stable prime divisors of }G/H\}.\]
In our case, we have a right action of $H$ on $\p^1 \simeq B \backslash G$ and this set identifies with $\p^1/H$ via the map:
\begin{align*}
    \p^1/H & \xrightarrow{\sim} \D^B \\
    [\alpha : \beta]H & \mapsto Z(\alpha w + \beta x)H,
\end{align*}
where $Z(\alpha w + \beta x)$ denotes the zero locus of the regular function $\alpha w + \beta x$ in $G$.\\

\item \emph{The invariant valuations.} We now describe the set 
\[ \V^G = \V^G(G/H) = \{G\text{-invariant geometric valuations of }\C(G/H)\}.\]
We start with the case where $H$ is trivial. A $G$-invariant geometric valuation on $\C(G)$ is determined by its restriction on the subset
\[ \p^1 \simeq \{ \alpha w + \beta x \: | \: [\alpha:\beta] \in \p^1\} \subset \C[G]. \]
Furthermore, we can show that for any $\nu \in \V^G$, there exists $[\alpha_0 : \beta_0] \in \p^1$ and $ -1 \leq r = \frac{p}{q} \leq 1$ (with $p$ and $q$ relatively prime and $q \geq 1$) such that $\nu(\alpha_0 w + \beta_0 x) = p$ and $\nu(\alpha w + \beta x) = -q$ for all $[\alpha : \beta] \neq [\alpha_0 : \beta_0]$. Therefore,
\[\V^G \simeq \{\p^1 \times ([-1,1] \cap \mathbb{Q})\} / \sim\]
where $([\alpha : \beta],-1) \sim ([0 : 1],-1)$ for all $[\alpha : \beta] \in \p^1$. For $j \in \p^1$ and $r \in [-1:1] \cap \mathbb{Q}$, we denote by $\nu(j,r)$ the $G$-invariant geometric valuation corresponding to $(j,r)$. The valuation corresponding to $(j,-1)$ does not depend on $j$; we denote this valuation by $\nu_0$.

Now for any finite subgroup $H \subset G$, any $G$-invariant geometric valuations on $\C(G)^H$ extends to a valuation on $\C(G)$. Thus we can also associate a couple $(j,r)$ with $j \in \p^1/H$ and $r \in [-1,1] \cap \mathbb{Q}$ to a $G$-invariant geometric valuation on $\C(G)^H$. We again denote by $\nu(j,r)$ the valuation of $\V^G(G/H)$ corresponding to $j \in \p^1/H$ and $r \in [-1,1] \cap \mathbb{Q}$.

We let $b(j) \in \mathbb{Q}$ be the maximum rational number $r$ such that $\V^G(G/H)$ has a valuation of the form $\nu(j,r)$. For example, for $H = A_k$ we have 
\begin{align*}
    & b([0:1]) = b([1:0]) = 1 \\
    & b(j) = \frac{4}{m}-1 \text{ if } j\notin \{[0:1],[1:0]\}
\end{align*}
where $m = n$ if $n$ is even, and $m = 2n$ if $n$ is odd.\\

\item \emph{The colored data.} Now let $X$ be a $G$-variety in the $G$-birational class of $G/H$, and let $\phi : X_0 \dashrightarrow X$ be a $G$-equivariant birational map. The pair $(X,\phi)$ is called a $G$-model of $X_0$. In the general theory, we associate to each $G$-orbit $Y \subset X$ a pair
\begin{align*}
\bullet \: & \V^G_Y = \{ \nu_D \in \V^G(X) \: | \text{ $Y \subset D$, where $D$ is a $G$-stable prime divisor of $X$}\} \\ &\subset \V^G(X) \simeq \V^G;\text{ and}\\
\bullet \:  & \D^B_Y = \{D \in \D^B(X) \: | \: Y \subset D\} \subset \D^B(X) \simeq \D^B.
\end{align*}
Such a pair is called the colored data of the $G$-orbit $Y$, and the sets of all those pairs is called the colored equipment of $X$.

We give a description of all possible colored data in our particular case, i.e. where $X$ is $G$-equivariant embedding of $G/H$ with $H \subset G = \SL_2$ a finite subgroup. We call the facet of an orbit the set of valuations in $\V_1(G/H) = \V^G(G/H) \backslash \{\nu_0\}$ which dominate the local ring of the closure of the given orbit.

The (non-open) $G$-orbits are divided into the six following types:\smallskip \begin{itemize}[leftmargin=7mm]
\item[$\bullet$] {\bf Type $\mathcal{C}(j,r)$} ($2$-dimensional orbit): $j\in\p^1/H$ and $r\in\mathbb{Q}\cap]-1,b(j)]$.\\ For this case, $\V_Y^G=\{\nu(j,r)\}$ is a single element of $\V_1$, 
 $\D_Y^B=\varnothing$, and the facet of $Y$ is $\{\nu(j,r)\}$.\smallskip

\item[$\bullet$]{\bf Type $\mathcal{A}_N(j_1,\ldots,j_N,r_1,\ldots,r_N)$} (orbit isomorphic to $\p^1$):  $j_1,\ldots,j_N$ are  $N$ distinct elements of $\p^1/H$, and $r_i\in\mathcal{Q}\cap]-1,b(j_i)]$ for each $1 \leq i \leq N$.\\
Suppose also that $\sum_{i=1}^N\frac{1}{1+r_i}\ge1$.  Then $\V_Y^G=\{\nu(j_1,r_1),\ldots,\nu(j_N,r_N)\}$ and $\D_Y^B=\p^1/H\setminus\{j_1,\ldots,j_N\}$. The facet of $Y$ is 
\[
\bigcup_{i=1}^N\{\nu(j_i,s_i)\ |\ -1<s_i<r_i\}\cup \{\nu(j,s)\ |\ j\not\in\{j_1,\ldots,j_N\},-1<s\le b(j)\}.
\]\smallskip

\item[$\bullet$]{\bf Type $\mathcal{AB}(j,r_1,r_2)$} (orbit isomorphic to $\p^1$): $j\in\p^1/H$, and $r_1,r_2$ are two rational numbers satisfying $-1\le r_1<r_2\le b(j)$. Then $\V_Y^G=\{\nu(j,r_1), \nu(j,r_2)\}$ and $\D_Y^B=\varnothing$. The facet of $Y$ is $\{\nu(j,s)\ |\ r_1<s<r_2\}$.\smallskip

\item[$\bullet$]{\bf Type $\mathcal{B}_+(j,r)$} (orbit isomorphic to $\p^1$): $j\in \p^1/H$, and $r\in\mathcal{Q}\cap[-1,b(j)[$. Then $\V_Y^G=\{\nu(j,r)\}$ and $\D_Y^B=\{j\}$. The facet of $Y$ is $\{\nu(j,s)\ |\ r<s\leq b(j)\}$.\smallskip

\item[$\bullet$]{\bf Type $\mathcal{B}_-(j,r)$} (orbit isomorphic to $\p^1$): $j\in \p^1/H$, and $r\in\mathcal{Q}\cap]0,b(j)[$. Then $\V_Y^G=\{\nu(j,r))\}$ and $\D_Y^B=\p^1/H\setminus\{j\}$. The facet of $Y$ is $\{\nu(j,s)\ |\ r<s\leq b(j)\}$.\smallskip

\item[$\bullet$] {\bf Type $\mathcal{B}_0(j,r)$} (fixed point): $j\in \p^1/H$, and $r\in\mathcal{Q}\cap]0,b(j)[$. Then $\V_Y^G=\{\nu(j,r))\}$ and $\D_Y^B=\p^1/H$. The facet of $Y$ is $\{\nu(j,s)\ |\ r<s\leq b(j)\}$.

\end{itemize}

The facets of the colored data of a $G$-model do not intersect, and their unions is closed in $\V^G$. Therefore, we can now represent the colored equipment of equivariant embeddings of homogeneous $\SL_2$-threefolds. We represent the set $\V^G(G/H)$ by a skeleton diagram, and we represent each orbit by its facet. Since the facets for orbits of types $\mathcal{B}_{-}$, $\mathcal{B}_0$ and $\mathcal{B}_{+}$ are the same, we differentiate them with a sign $0$, $+$ or $-$. See the \hyperref[appendix: A]{Appendix} for many examples of such diagrams.\\

\item \emph{Luna-Vust over the field of real numbers.} Finally, we present how the Luna-Vust theory allows us to find real forms of equivariant embeddings of $\SL_2/H$. To do this, Moser-Jauslin and Terpereau defined the following action of the Galois group $\Gamma$ on the colored equipment of $G/H$ in \cite[\S~4.3]{MJT20}, generalizing the action defined by Huruguen \cite{Hur11} in the case of spherical varieties.

Let $\sigma \in \{\sigma_s, \sigma_c\}$ and let $\mu$ be a $(G,\sigma)$-equivariant real structure on $G/H$. By Proposition \ref{prop:Z}, we have $\mu(gH) = \sigma(g)tH$ for a certain $t \in G$. Then the action of the Galois group $\Gamma$ on the set of colors $\D^B(G/H) \simeq \p^1/H$ is given by
\[\left\{ \begin{array}{ll}
    [\alpha : \beta]H \mapsto [\overline{\alpha} : \overline{\beta}]tH & \text{if $\sigma = \sigma_s$; and} \\[1pt] 
    [\alpha : \beta]H \mapsto [-\overline{\beta} : \overline{\alpha}]tH & \text{if $\sigma = \sigma_c$}
\end{array} \right. \]
and the $\Gamma$-action on $\V^G(G/H)$ is defined by
\[ \forall j \in \p^1/H, \forall \nu(j,r) \in \V^G(G/H), \gamma \cdot \nu(j,r) = \nu(\gamma \cdot j,r).\]
Note that this action preserves the type of the $G$-orbits. It can be visualised on the skeleton diagrams where it corresponds to a permutation of the spokes.\\

We can now give the following result that determines when an equivariant real structure on a homogeneous space $X_0$ extends to an effective equivariant real structure on an equivariant embedding of $X_0$, using the previously defined $\Gamma$-action on the colored equipment of that embedding.

\begin{theorem}  \emph{(\cite[Theorem 4.7]{MJT20}\label{th:E})}
Let $H$ be a finite subgroup of $G=\SL_2$, let $\sigma$ be a real group structure on $G$ corresponding to the real form $F$ of $G$, and
let $\mu$ be a $(G,\sigma)$-equivariant real structure on $X_0=G/H$.
Let $X_0 \hookrightarrow X$ be a $G$-equivariant embedding of $X_0$. 
Then $\mu$ extends to an effective $(G,\sigma)$-equivariant real structure $\tilde{\mu}$ on $X$, which corresponds to an $(\R,F)$-form of $X$ through the map $X \mapsto X/\left\langle \tilde{\mu} \right \rangle$, if and only if
\begin{enumerate}[leftmargin=10mm]
\item the $\Gamma$-actions on $\V^G(G/H)$ and $\D^B(G/H)$ induced by $\mu$ preserves the collection of colored data of the $G$-orbits of $X$; and
\item every $G$-orbit of $X$ of type $\mathcal{B}_0$ or $\mathcal{B}_-$ is stabilized by the $\Gamma$-action.
\end{enumerate}
\end{theorem}
\end{itemize}

\section{Equivariant automorphism groups} \label{section: group}

In this section, we first see how we can use the Luna-Vust theory to compute the equivariant automorphism groups of equivariant embeddings of homogeneous spaces. Then, we apply to the particular case of the minimal smooth completions of $\SL_2/A_k$ listed in Section \ref{section: setting}.

Let $k$ be an algebraic closed field, $G$ a reductive group over $k$, $H \subset G$ an algebraic subgroup of $G$ and $B \subset G$ a Borel subgroup of $G$. We denote by $K$ the field $k(G)^H = k(G/H)$.

We have an isomorphism (see \cite[Proposition 1.8]{Tim11})
\begin{align*}
	N_G(H)/H & \longrightarrow Aut^G(G/H)  \\
	nH & \longmapsto ( \varphi_n : kH \mapsto kn^{-1}H).
\end{align*}

We consider two group actions on $G/H$ :
\begin{itemize}
    \item The action of $A = Aut^G(G/H) \cong N_G(H)/H$
    \item The action of $G$ by left multiplication.
\end{itemize}

Those actions induce actions of $A$ and $G$ on $K$. Indeed, for $f \in K$, $n \in N_G(H)$ and $g \in G$, we define:
\begin{itemize}
    \item $(g \cdot f)(x) = f(g^{-1}x)$ and,
    \item $(\varphi_n(f))(kH) = f(\varphi_n^{-1}(kH)) = f(knH).$
\end{itemize}

We can now define an action of $A$ on $\V^G$ and on $\D^B$.

\begin{itemize}
    \item Action of $A$ on $\V^G$: If $\nu \in \V^G$, and if $n \in N_G(H)$, then $((\varphi_n \nu)(f)) = \nu(\varphi_n^{-1}(f)).$ Since $G$ acts on $K$ on the left, this does indeed give a $G$-stable geometric valuation on $G/H$.

\item Action of $A$ on $\D^B$: Again, since $B$ acts on the left and the action of $A$ is induced by a right action of $N_G(H)$, the action of $A$ on $G/H$ induces an action on $B$-stable divisors.
\end{itemize}

Let $X_0 = G/H$. We can now determine when a $G$-equivariant automorphism of $X_0$ extends to an automorphism on a $G$-model of $X_0$.

\begin{theorem}\label{prop : Aut}
	Let $(X,\psi)$ be a $G$-model of $X_0$, and $\varphi \in A$. Let $\mathbb{F}(X,\psi) = \{(\V^G,\D^B)\}$ be the colored data of $(X,\psi)$. Then, the automorphism $\varphi$ extends to an automorphism of $X$ if and only if $\mathbb{F}(X,\psi)$ is preserved by the action of $\varphi$ on $\V^G$ and $\D^B$, i.e.
	\[\forall (\V^G_Y,\D^B_Y) \in \mathbb{F}(X,\psi), (\varphi(\V^G_Y),\varphi(\D^B_Y)) \in \mathbb{F}(X,\psi).\]
\end{theorem} 

For the proof, we will need the following result that gives a combinatorial description of the $B$-charts of $X$. We recall that a $B$-chart of a $G$-variety $X$ is a $B$-stable affine open subset of $X$.

\begin{proposition}\cite[Corollary 13.9]{Tim11} \label{prop}
	The map 
	\[ \{\text{$B$-charts of }X\} \to \mathfrak{P}(\V_{X}^{G}) \times \mathfrak{P}(\D_{X}^{B}),\ \mathring{X} \mapsto  
	\bigcup_{\substack{\text{$G$-orbits $Y \subseteq X$},\\ Y \cap \mathring{X} \neq \varnothing}} \left(  \V_{Y}^{G}, \D_{Y}^{B} \right) \]
	is a bijection between the $B$-charts of $X$ and the pairs $(\W,\RR) \in   \mathfrak{P}(\V_{X}^{G}) \times \mathfrak{P}(\D_{X}^{B})$ satisfying conditions $(C)$, $(F)$, and $(W)$ (See \cite[Chapter 13]{Tim11}).
	
	The converse map sends the pair $(\W,\RR)$ to $\Spec(A[\W,\RR])$, where
	\begin{equation*}
		A[\W,\RR]= \left(\bigcap_{w \in \W} \O_w\right) \cap  \left(\bigcap_{D \in \RR} \O_{\nu_D}\right) \cap  \left(\bigcap_{D \in \D \setminus \D^B} \O_{\nu_D}\right) \subseteq  \k(X)
	\end{equation*} 
	and $\D$ is the set of prime divisors on $X$ that are not $G$-stable. 
\end{proposition}

\begin{proof}[Proof of the theorem]
	Let's suppose that $\varphi$ extends to an automorphism of $X$. Let $Y$ be a $G$-orbit of $X$, then $\varphi(Y)$ is also a $G$-orbit of $X$. We have $\varphi(\D^B_Y) = \D^B_{\varphi(Y)}$, and $\varphi(\V^G_Y) = \V^G_{\varphi(Y)}$.
	Therefore, $(\varphi(\V^G_Y),\varphi(\D^B_Y)) \in \mathbb{F}(X,\psi)$ and the colored data is preserved.
	
	Conversely, assume that the colored data $\mathbb{F}(X,\psi)$ is preserved by the action of $\varphi$. Let $Y$ be a $G$-orbit of $X$, and let $\mathring{X}$ be a $B$-chart of $X$ such that $\mathring{X} \cap Y \neq \varnothing$. By Proposition $\ref{prop}$, there exists a unique pair $(\W,\RR)$ such that $\Spec(A[\W,\RR])= \mathring{X}$. Now, the pair $(\varphi^{-1}(\W),\varphi^{-1}(\RR))$ also satisfies conditions (C), (F) and (W). Since $\mathbb{F}(X,\psi)$ is stable under $\varphi$, this means that $\mathring{X}' = \Spec(A[\varphi^{-1}(\W),\varphi^{-1}(\RR)])$ is a $B$-chart of X. 
	
	Now, $\varphi$ sends $A[\varphi^{-1}(\W),\varphi^{-1}(\RR)]$ to $A[\W,\RR]$ so $\varphi$ is a regular map $\mathring{X} \rightarrow \mathring{X}'$. Indeed, let $f \in A[\varphi^{-1}(\W),\varphi^{-1}(\RR)]$ we have :
	\[ \forall \omega \in \W, \omega(\varphi(f)) = (\varphi^{-1}(\omega))(f) \geq 0 \text{ and} \]
	\[\forall D \in \RR, \nu_D(\varphi(f)) = \nu_{\varphi^{-1}(D)}(f) \geq 0.\]
	Moreover, $\D \backslash \D^B$ is preserved by $A$ so $\varphi(f) \in A[\W,\RR]$ by Proposition \ref{prop}.
	
	Thus, the open locus $U_{\varphi} \subset X$ over which $\varphi$ is defined contains $\mathring{X}$. Since, $\varphi$ is $G$-equivariant, $U_{\varphi}$ actually contains $G \cdot \mathring{X}$, in particular it contains $Y$. This is true for all $G$-orbit $Y$ of $X$ so $\varphi$ is a regular morphism on $X$. Moreover, since $\varphi^{-1} \in A$ and $\varphi^{-1} \cdot \mathbb{F}(X,\psi) = \varphi^{-1} \cdot ( \varphi \cdot \mathbb{F}(X,\psi)) = \mathbb{F}(X,\psi)$, $\varphi^{-1}$ is also a regular morphism on $X$. Therefore, $\varphi$ is an automorphism of $X$.
\end{proof}
We now let $k = \C$, $G = \SL_2$, $H$ be a finite cyclic subgroup of $G$ and $X$ be a minimal smooth completion of $G/H$. The diagrams corresponding to those varieties can be found in the \hyperref[appendix: A]{Appendix} and the groups $\Aut^G(X_0)$ have been computed in \cite[Lemma 4.3]{MJT20} for all $H$. We do not give the equivariant automorphism group in cases (b1) and (b2), because the result depends heavily on the different parameters and is not required for the rest of this paper.

\begin{corollary}\label{prop : Aut 2}
\begin{itemize} \item []
        \item  If $X$ corresponds to model (c1), (g3) or (i4) with $n = m$, we have \\$\Aut^G(X) = \begin{bmatrix} * & 0 \\ 0 & * \end{bmatrix}/H$.
        \item If $X$ corresponds to model (c1), (g3) or (i4) with $n \neq m$, we have\\ $\Aut^G(X) = \left(\begin{bmatrix} * & 0 \\ 0 & * \end{bmatrix} \cup \begin{bmatrix} 0 & * \\ * & 0 \end{bmatrix}\right)/H$.
        \item For the other varieties, the equivariant automorphism group are given in the following table.
        \end{itemize}
\end{corollary}
\[\tabulinesep=0.8mm
\begin{tabu}{|c|c|}
	\hline 
	\text{Models} & \Aut^G(X)  \\ \hline 
	\text{(a1)}  & \SL_2 \\
	\hline 
	\text{(a2)} & \PGL_2 \\
	\hline
	\text{(d1),(f1),(c2)} &
	\begin{bmatrix}
	    * & * \\ 0 & *
	\end{bmatrix}/H \\
    \hline
    \text{(e1),(h4)} & \left(\begin{bmatrix} * & 0 \\ 0 & * \end{bmatrix} \cup \begin{bmatrix} 0 & * \\ * & 0 \end{bmatrix}\right)/H \\
	\hline
    \begin{array}{c}
	\text{(g1),...,(k1),(d2),(e2),(f2)} \\
    \text{(a3),...,(f3),(a4),...,(f4)} \end{array} & \begin{bmatrix}
	    * & 0 \\ 0 & *
	\end{bmatrix}/H\\
	\hline
    \text{(g2)}  & \mathfrak{S}_3 \\
	\hline
    \text{(g4)}  & \Z/2\Z \\
	\hline
    \text{(g4')}  & \{1\} \\
	\hline
\end{tabu}\]

\begin{proof}
    We use Theorem \ref{prop : Aut}. Let $n = \begin{bmatrix} a & b \\ c & d \end{bmatrix} \in N_G(H)$. 
    For cases (a1) and (a2), the diagram is preserved by the action of $\varphi_n$ for all $n$ so we have $\Aut^G(X_0) = \Aut^G(X)$, hence the result.
    
    We have $\varphi_n \cdot [0:1]H = [0:1]n^{-1}H = [-c:a]H$ and $
        \varphi_n \cdot [1:0]H = [d:-b]H$.
    Therefore, $[0:1]$ (resp. $[1:0]$) is fixed by the action of $\varphi_n$ if and only if $c = 0$ (resp. $b=0$), and $[0:1]$ and $[1:0]$ are permuted if and only if $a = d = 0$. This gives us the result for all models except (g2), (g4) and (g4').

    For model (g2), the equivariant automorphism group has been computed in \cite[Example 4.11]{MJT20} so we are left with models (g4) and (g4'). We can now suppose that $k \geq 4$ and is even and so we have $N_G(H)/H =\left(\begin{bmatrix} * & 0 \\ 0 & * \end{bmatrix} \cup \begin{bmatrix} 0 & * \\ * & 0 \end{bmatrix}\right)/H$. Let $\alpha \in \C^*$, $\varphi_n \in \Aut^G(X)$ and suppose that $nH = \begin{bmatrix}
        a & 0 \\ 0 & a^{-1}
    \end{bmatrix}H$. For (g4'), $[0:1]$ and $[1:0]$ both have to be fixed so $n$ must be of this form. Moreover, we have $\varphi_n \cdot [1:\alpha]H = [1:a^2\alpha]H$ so $[1:\alpha]$ is fixed if and only if $a = \pm 1$, i.e. if and only if $nH = I_2H$.
    
    Now suppose that $nH = \begin{bmatrix}
        0 & b \\ -b^{-1} & 0
    \end{bmatrix}H$. Then $\varphi_n \cdot [1:\alpha] = [1: -b^2 \alpha^{-1}]$ so $[1:\alpha]$ is fixed by the action of $\varphi_n$ if and only if $b = \pm i \alpha$, i.e. $nH = \begin{bmatrix}
        0 & i\alpha \\ i \alpha^{-1} & 0
    \end{bmatrix}H$.
\end{proof}

\section{Determination of real forms} \label{section: forms}

Let $G =\SL_2$, $H$ a finite cyclic subgroup of $G$ of order $k$, $X_0 = G/H$, $\Gamma=\Gal(\C/\R) = \{1, \gamma\}$, and $\sigma \in \{\sigma_s, \sigma_c\}$. $H$ is conjugate to the cyclic subgroup $A_k$ of $\SL_2$ generated by $\omega_k = \begin{bmatrix}
	\zeta_k & 0 \\ 0 & \zeta_k^{-1}
\end{bmatrix}$ where $\zeta_k$ is a primitive $k$-th root of unity. For the rest of this section, we will suppose that $H = A_k$. We also note $e = \begin{bmatrix}
    0 & 1 \\ -1 & 0
\end{bmatrix} \in \SL_2(\C)$ et $f = \begin{bmatrix}
    0 & i \\ i & 0
\end{bmatrix} \in \SL_2(\C)$.

In this section, we will classify the equivariant real structures of the minimal smooth completions of $X_0$ and their equivalence classes. We recall that the diagrams of those models are in the \hyperref[appendix: A]{Appendix}. We start by the case where $H$ is trivial.

\begin{proposition} \label{prop 4.1}
	Let $X$ be a minimal smooth completion of $G = \SL_2$, and let $\sigma \in \{ \sigma_s, \sigma_c\}$. Then the following table gives, in each case, a representant $\mu$ for each equivalence class of effective $(G,\sigma)$-equivariant real structures on $X$. 
\end{proposition}

\[\tabulinesep=0.8mm
\begin{tabu}{|c|c|c|}
	\hline 
	\text{Models} & \mu \text{ for } \sigma=\sigma_s & \mu \text{ for } \sigma=\sigma_c \\ \hline 
	\text{(a1)}  & g \mapsto \sigma_s(g) & \begin{array}{c}
		g \mapsto \sigma_c(g) \\ g \mapsto -\sigma_c(g)
	\end{array}\\
	\hline
	\text{(c1),(d1),(f1)(g1),(h1),(i1),(j1),(k1)} & g \mapsto \sigma_s(g) & \varnothing \\
	\hline
		\text{(e1)} & \begin{array}{c}
			g \mapsto \sigma_s(g) \\ g \mapsto \sigma_s(g)f
		\end{array}  & \begin{array}{c}
		g \mapsto \sigma_c(g) \\ g \mapsto -\sigma_c(g)
	\end{array}\\
	\hline
		\begin{array}{c}
	\text{(b1) and all the $x_i$} \\ \text{are congruent modulo $\pi$}
	\end{array} &
	g \mapsto \pm \sigma_s(g)\begin{bmatrix}
		e^{-i\pi x} & 0 \\ 0 & e^{i\pi x}
	\end{bmatrix} &
	\varnothing\\
	\hline
	\begin{array}{c}
		\text{(b1) and not all the $x_i$} \\ \text{are congruent modulo $\pi$}
	\end{array} &
	\varnothing &
	\varnothing\\
	\hline
\end{tabu} \]

\begin{proof}
\textit{- Step 1:} We start by looking for all the $(G, \sigma)$-equivariant real structures on $X_0$. For this, we use Proposition \ref{prop:Z} and so we just have to compute $Z^1(\Gamma, N_G(H)/H)$. In this case, $N_G(H)/H = \SL_2$ (\cite[Lemma 4.3]{MJT20}) and since $\Gamma \simeq \Z/2\Z$ we have
\[Z^1(\Gamma,G) = \{ g \in G \ | \   g^{-1}= \gamma \cdot g \}.\]
Let $g =\begin{bmatrix}
	a & b \\ c & d
\end{bmatrix} \in \SL_2$. We start by the case $\sigma=\sigma_s$. Then, $\gamma \cdot g = \begin{bmatrix}
	\bar{a} & \bar{b} \\ \bar{c} & \bar{d}
\end{bmatrix}$, where $\bar{x}$ denotes the complex conjugation of $x$ for $x \in \C$, and $g^{-1} = \begin{bmatrix}
d & -b \\ -c & a
\end{bmatrix}$ so we get
\begin{align*}
	\gamma \cdot g = g^{-1} & \Leftrightarrow \bar{a}=d,\ \bar{d}=a,\ c= -\bar{c},\ b= -\bar{b}\\	
	& \Leftrightarrow \bar{a}=d \ \et \ b,c \in i\R.				
\end{align*}
Therefore, 
\[Z^1(\Gamma,G) = \left\{\begin{bmatrix}
		a & ix \\ iy & \bar{a}
	\end{bmatrix}; \ x,y \in \R, \  |a|^2 + xy = 1\right\}.\]
For $\sigma = \sigma_c$, $\gamma \cdot g = {}^{t}\sigma_s(g)^{-1} = \begin{bmatrix}
	\bar{d} & -\bar{c} \\ -\bar{b} & \bar{a}
\end{bmatrix}$ so						
\begin{align*}
	\gamma \cdot g = g^{-1} & \Leftrightarrow  \bar{d}=d,\ \bar{a}=a,\ b=\bar{c},\ c=\bar{b}\\
	& \Leftrightarrow c=\bar{b} \ \et \ a,d \in \R.							
\end{align*}
which gives
\[Z^1(\Gamma,G) = \left\{\begin{bmatrix}
		a & b \\ \bar{b} & d
	\end{bmatrix}; \ a,d \in \R, \  ad - |b|^2 = 1\right\}\]

\textit{- Step 2:} We will now see in each case, using Theorem \ref{th:E} which one of those equivariant real structures on $X_0$ extend to an equivariant real structure on $X$.

First, for model (a1), the result is immediate in both cases because the diagram is necessarily preserved by the $\Gamma$-action on it. Let us also notice that except for model (b1), we only need to look at whether the action of $\Gamma$ fixes or permutes $[0:1]$ and $[1:0]$.
	
	$\bullet$ Let us first look at the case where $\sigma = \sigma_s$. Let $t = \begin{bmatrix}
		a & ix \\ iy & \bar{a}
	\end{bmatrix} \in Z^1(\Gamma,G)$.
\begin{center}
	$\gamma \cdot [0:1]=[0:1]t = [iy:\bar{a}]$ and $\gamma\cdot [1:0]=[1:0]t = [a:ix]$.
\end{center} So,
\begin{align*}
	& \gamma \cdot [0:1]= [0:1] \Leftrightarrow y=0 \ \et \ a \neq 0 \\
	\et \ & \gamma\cdot [1:0]=[1:0] \Leftrightarrow x=0 \ \et \ a \neq 0.
\end{align*}
Therefore, the action defined by $\mu_t$ fixes $[0:1]$ and $[1:0]$ if and only if $t = \begin{bmatrix}
	a & 0\\ 0 & a^{-1}
\end{bmatrix}$ with $|a| = 1$.

Furthermore,
\begin{align*}
	& \gamma \cdot [0:1]= [1:0] \Leftrightarrow y \neq 0 \ \et \ a=0 \\
	\et \ & \gamma\cdot [1:0]=[0:1] \Leftrightarrow x \neq 0 \ \et \ a=0.
\end{align*}
So $\gamma$ permutes $[1:0]$ and $[0:1]$ if and only if $t = \begin{bmatrix}
	0 & ix \\ ix^{-1} & 0
\end{bmatrix}$ with $x \in \R^*$.

We can now conclude for all models other than model (b1). Indeed, for models (d1) and (f1), $\gamma$ has to fix $[0:1]$ for $\mu_t$ to extend effectively to $X$. For models (c1), (g1), (h1), (i1), (j1) and (k1), $\gamma$ has to fix both $[0:1]$ and $[1:0]$. And for model (e1), $\gamma$ can either permute $[0:1]$ and $[1:0$] or fix them.

Now let $m \geq 3$ and consider model (b1). We note $(\w_1,...,\w_n)$ the elements of $\p^1$ corresponding to the orbits of type $\mathcal{B}_-$, with $\w_1 = [0:1]$, $\w_n = [1:0]$ and $\w_i = [1:\alpha_i]$ for $i \in \{2,...,m-1\}$. We also note $x_i = \text{arg}(\alpha_i)$. For an equivariant real structure on $X_0$ to extend to an effective equivariant real structure on $X$, all of the $\w_i$ have to be fixed by the corresponding action. For $\w_1$ and $\w_n$ to be fixed, we already know that $t$ must be of the form $\begin{bmatrix}
	a & 0 \\ 0 & a^{-1}
\end{bmatrix}$. For such a $t$, $\gamma \cdot \w_i = [1 : a^{-2} \overline{\alpha_i}]$ so $\gamma \cdot \w_i = \w_i$ if and only if $a = \pm e^{-i\pi x_i}$. Therefore we have a $(G,\sigma_s)$-equivariant real structure on $X$ if and only if all of the $x_i$ are congruent modulo $\pi$, in which case we actually have two real structures corresponding to $t = \pm \begin{bmatrix}
e^{-i\pi x} & 0 \\ 0 & e^{i\pi x}
\end{bmatrix}$.

Let us now look at the case $\sigma = \sigma_c$. Let $t = \begin{bmatrix}
	a & b \\ \bar{b} & d
\end{bmatrix} \in Z^1(\Gamma,G)$. The conditions for $\mu_t$ to extend effectively to $X$ are the same as in the case $\sigma = \sigma_s$, but the action is now different :
\begin{align*}
	& \gamma \cdot [0:1] = [-1:0]t = [-a:-b] = [a:b]\\
	\et \ & \gamma \cdot [1:0] = [0:1]t = [\bar{b}:d].
\end{align*}
So $\gamma \cdot [0:1] = [0:1] \Rightarrow a = 0$, which is impossible since $t \in \SL_2$. This shows that $[0:1]$ cannot be fixed by $\gamma$. Notice that this is enough to conclude for model (b1) too.
Then, we also have that 
\begin{align*}
	& \gamma \cdot [0:1]= [1:0] \Leftrightarrow a \neq 0 \ \et \ b=0 \\
	\et \ & \gamma\cdot [1:0]=[0:1] \Leftrightarrow d \neq 0 \ \et \ b=0.
\end{align*}
So the action defined by $\mu_t$ permutes $[0:1]$ and $[1:0]$ if and only if $t = \begin{bmatrix}
	a & 0 \\ 0 & a^{-1}
\end{bmatrix}$ with $a \in \R^*$.

To recap, the following table gives in each case for which $t \in Z^1(\Gamma,G)$ the corresponding equivariant real structure $\mu_t : g \mapsto \sigma(g)t$ on $X_0 = G$ extends to an effective ($G,\sigma$)-equivariant real structure on $X$:

\vspace{3mm}
\scalebox{0.7}{$\tabulinesep=0.8mm
\hspace{-8mm} \begin{tabu}{|c|c|c|}
		\hline 
		\text{Models} & \sigma=\sigma_s & \sigma=\sigma_c\\ \hline 
		\text{(a1)} & \left\{\begin{bmatrix}
			a & ix \\ iy & \bar{a}
		\end{bmatrix}; \ x,y \in \R, \  |a|^2 + xy = 1\right\} & \left\{\begin{bmatrix}
			a & b \\ \bar{b} & d
		\end{bmatrix}; \ a,d \in \R, \  ad - |b|^2 = 1\right\}\\
		\hline
		\text{(e1)}   & \left\{\begin{bmatrix}
			a & 0 \\ 0 & a^{-1}
		\end{bmatrix}; \ |a| = 1\right\} \sqcup \left\{\begin{bmatrix}
		0 & ix \\ ix^{-1}  & 0
	\end{bmatrix}; \ x \in \R \right\} & \left\{\begin{bmatrix}
	a & 0 \\ 0 & a^{-1}
	\end{bmatrix}; \ |a| = 1\right\} \\
		\hline
		\text{(d1) and (f1)}  & \left\{\begin{bmatrix}
			a & ix \\ 0 & a^{-1}
		\end{bmatrix}; \ x \in \R, \ |a| = 1\right\} &  \varnothing \\
		\hline
		\text{(c1), (g1), (h1), (i1), (j1) and (k1)}  & \left\{\begin{bmatrix}
			a & 0 \\ 0 & a^{-1}
		\end{bmatrix}; \ |a| = 1\right\} &  \varnothing \\
		\hline
		\begin{array}{c}
	\text{(b1) and all the $x_i$} \\ \text{are congruent modulo $\pi$}
	\end{array} &
	\pm \begin{bmatrix}
		e^{-i\pi x} & 0 \\ 0 & e^{i\pi x}
	\end{bmatrix} &
	\varnothing\\
	\hline
	\begin{array}{c}
		\text{(b1) and all the $x_i$} \\ \text{are not congruent modulo $\pi$}
	\end{array} &
	\varnothing &
	\varnothing\\
		\hline	
	\end{tabu}$}\\

\vspace{3mm}
\textit{- Step 3:} 	Finally, we determine the equivalence classes of the effective equivariant real structures for each model using the results of Corollary \ref{prop : Aut 2}.

For model (a1), $\Aut_{\C}^G(X) = \Aut_{\C}^G(X_0)$, so we have the same results as in the homogeneous case.
	
Suppose that $X$ corresponds to diagrams (c1),(g1),(h1),(i1),(j1) or (k1). Then we have no $(G,\sigma_c)$-equivariant real structures, so we suppose that $\sigma = \sigma_s$. Let $t = \begin{bmatrix}
		s & 0 \\ 0 & s^{-1}
	\end{bmatrix}$ with $|s| = 1$ and $n =\begin{bmatrix}
	x & 0 \\ 0 & x^{-1}
\end{bmatrix}$ with $x$ such that $x^{-2} = s$. By Corollary \ref{prop : Aut 2}, $\varphi_n \in \Aut_{\C}^G(X)$. For all $g \in G$, we have $\varphi_n \circ \mu_{Id} \circ \varphi_n^{-1}(g) =  \sigma_s(g)t = \mu_t(g)$ so $\mu_{Id}$ is equivalent to $\mu_t$. Therefore, all $(G,\sigma_s)$-equivariant real structures are therefore equivalent to $\mu_{Id}$.

Suppose now that $X$ corresponds to (d1) or (f1). We have no equivariant real structures for $\sigma = \sigma_c$. Let $\sigma = \sigma_s$, $t = \begin{bmatrix}
	s & ix \\ 0 & s^{-1}
\end{bmatrix}$ with $|s| = 1$ and $x \in \R$. Let $\alpha,\beta \in \C$ such that $\alpha^{-2} = s$ and $\text{Im}(\alpha \bar{\beta}) = \frac{x}{2}$, and let $n = \begin{bmatrix}
\alpha & \beta \\ 0 & \alpha^{-1}
\end{bmatrix} \in \Aut_{\C}^G(X)$. Then we have $\varphi_n \circ \mu_{Id} \circ \varphi_n^{-1} = \mu_t$ so all $(G,\sigma_s)$-equivariant real structures on $X$ are equivalent to the one given by the identity matrix.

For model (e1), we recall that $\Aut_{\C}^G(X) = \begin{bmatrix}
	* & 0 \\ 0 & *
\end{bmatrix} \sqcup \begin{bmatrix}
0 & * \\ * & 0
\end{bmatrix}$. Let $g = \begin{bmatrix}
a & b \\ c & d
\end{bmatrix} \in G$ and let $\sigma = \sigma_c$. For $n = \begin{bmatrix}
x & 0 \\ 0 & x^{-1}
\end{bmatrix} \in \Aut_{\C}^G(X)$ we have 
\begin{align*}
	\varphi_n \circ \mu_{Id} \circ \varphi_n^{-1}(g) & = \begin{bmatrix}
		\overline{ax}x^{-1} & \overline{bx^{-1}}x \\ \overline{cx}x^{-1} & \overline{dx^{-1}}x
	\end{bmatrix} \text{and} \\
	\varphi_n \circ \mu_{f} \circ \varphi_n^{-1}(g) & = \begin{bmatrix}
		i\overline{b}|x|^{-2} & i\overline{a}|x|^2 \\ i\overline{d}|x|^{-2} & i\overline{c}|x|^2
	\end{bmatrix}.
\end{align*}
For $t = \begin{bmatrix}
	s & 0 \\ 0 & s^{-1}
\end{bmatrix}$ with $|s| = 1$, $\mu_t(g) = \begin{bmatrix}
\overline{a}s & \overline{b}s^{-1} \\ \overline{c}s & \overline{d}s^{-1}
\end{bmatrix}$ so $\mu_t$ is equivalent to $\mu_{Id}$ (by taking $x^{-2} = s$), and for $t = \begin{bmatrix}
0 & iy \\ iy^{-1} & 0
\end{bmatrix}$ with $y \in \R$, $\mu_t(g) = \begin{bmatrix}
i\overline{b}y^{-1} & i\overline{a}y \\ i\overline{d}y^{-1} & i\overline{c}y
\end{bmatrix}$ so $\mu_t$ is equivalent to $\mu_f$ when $y < 0$. Now for $n = \begin{bmatrix}
x & 0 \\ 0 & x^{-1}
\end{bmatrix} \in \Aut_{\C}^G(X)$ we get

\begin{align*}
	\varphi_n \circ \mu_{Id} \circ \varphi_n^{-1}(g) & = \begin{bmatrix}
		\overline{ax}x^{-1} & \overline{bx^{-1}}x \\ \overline{cx}x^{-1} & \overline{dx^{-1}}x
	\end{bmatrix} \text{and} \\
	\varphi_n \circ \mu_{f} \circ \varphi_n^{-1}(g) & = - \begin{bmatrix}
		i\overline{b}|x|^{-2} & i\overline{a}|x|^2 \\ i\overline{d}|x|^{-2} & i\overline{c}|x|^2
	\end{bmatrix}.
\end{align*}
The second equality shows that $\mu_t$ is also equivalent to $\mu_f$ for $t = \begin{bmatrix}
	0 & iy \\ iy^{-1} & 0
\end{bmatrix}$ when $y > 0$. The previous computations show that $\mu_{Id}$ and $\mu_{f}$ are not equivalent.

Now let $\sigma = \sigma_c$, and fix $t = \begin{bmatrix}
	s & 0 \\ 0 & s^{-1}
\end{bmatrix}$ with $|s| = 1$. We have $\mu_t(g) = \begin{bmatrix}
s\overline{d} & -s^{-1}\overline{c} \\ -s\overline{b} & s^{-1}\overline{a}
\end{bmatrix}$. For $n = \begin{bmatrix}
x & 0 \\ 0 & x^{-1}
\end{bmatrix}$, $\varphi_n \circ \mu_{Id} \circ \varphi_n^{-1} = \begin{bmatrix}
|x|^{-2}\overline{d} & -|x|^2\overline{c} \\ -|x|^{-2}\overline{b} & |x|^2\overline{a}
\end{bmatrix}$. Therefore, $\mu_t$ is equivalent to $\mu_{Id}$ if $s > 0$, and to $-\mu_{Id}$ if $s<0$. Since those two real structures are not equivalent in the homogeneous case \cite{MJT20}, they are also not equivalent on $X$.

Finally for model (b1), the two classes we found when $\sigma = \sigma_s$ and all the $x_i$ are congruent modulo $\pi$ are equivalent if and only if $n = \begin{bmatrix}
	i & 0 \\ 0 & i
\end{bmatrix} \in \Aut_{\C}^G(X)$, which is not the case by \ref{prop : Aut}.
\end{proof}

\begin{remark}
	 Models (b1) and (c1) have other equivariant real structures, but those real structures are not effective, and therefore do not correspond to a real form of $X$.
\end{remark}

Using the same strategy we can also determine the real forms of the minimal smooth completions of $\SL_2/A_k$ when $k \geq 2$.

\begin{proposition} \label{prop 4.3}
	Let $k=2$, $\sigma \in \{\sigma_s,\sigma_c\}$ and $X$ be a minimal smooth completions of $\PGL_2$, Then the following table gives, in each case, a representant $\mu$ for each equivalence class of effective $(G,\sigma)$-equivariant real structures on $X$.
\end{proposition}

\[\tabulinesep=0.8mm
\begin{tabu}{|c|c|c|}
\hline 
\text{Models} & \mu \text{ for } \sigma=\sigma_s & \mu \text{ for } \sigma=\sigma_c \\ \hline 
\text{(a2)}  & \begin{array}{c}
	gH \mapsto \sigma_s(g)H \\ gH \mapsto \sigma_s(g)eH
\end{array} & \begin{array}{c}
	gH \mapsto \sigma_c(g)H \\ gH \mapsto \sigma_c(g)eH
\end{array}\\
\hline
\text{(c2),(d2),(e2),(f2)} & gH \mapsto \sigma_s(g)H & gH \mapsto \sigma_c(g)eH \\
\hline
\text{(g2)}  & \begin{array}{c}
	gH \mapsto \sigma_s(g)H \\ gH \mapsto \sigma_s(g)fH
\end{array} & \begin{array}{c}
	gH \mapsto \sigma_c(g)eH \\ gH \mapsto \sigma_c(g)efH
\end{array} \\ 
\hline
\begin{array}{c}
	\text{(b2 with $m\geq 3$) and all} \\ \text{the $x_i$ are congruent modulo $\pi$}
\end{array} &
g \mapsto \sigma_s(g)\begin{bmatrix}
	e^{-i\pi x} & 0 \\ 0 & e^{i\pi x}
\end{bmatrix}H &
g \mapsto \sigma_c(g)\begin{bmatrix}
	0 & e^{i\pi x} \\ -e^{-i\pi x} & 0
\end{bmatrix}H\\
\hline
\begin{array}{c}
	\text{(b2 with $m\geq 3$) and not all} \\ \text{the $x_i$ are congruent modulo $\pi$}
\end{array} &
\varnothing &
\varnothing\\ 
\hline
\end{tabu}\]

\begin{proof} As in the trivial case, we start by computing $Z^1(\Gamma,N_G(H)/H)$. Here $N_G(H)/H \simeq \PGL_2$ by \cite[Lemma 4.3]{MJT20} and we obtain
\begin{align*}
    Z^1(\Gamma,G/H) = & \left\{\begin{bmatrix}
		a & ix \\ iy & \bar{a}
	\end{bmatrix}H ; \ x,y \in \R, \  |a|^2 + xy = 1\right\} & \text{if $\sigma = \sigma_s$}\\ 
	\sqcup & \left\{\begin{bmatrix}
		a & b \\ c & -\bar{a}
	\end{bmatrix}H ; \ b,c \in \R, \  |a|^2 + bc = -1\right\} & \\
    Z^1(\Gamma,G/H) = & \left\{\begin{bmatrix}
		a & b \\ \bar{b} & d
	\end{bmatrix}H ; \ a,d \in \R, \  ad - |b|^2 = 1\right\} &  \text{if $\sigma = \sigma_c$} \\
	\sqcup & \left\{\begin{bmatrix}
		ix & b \\  -\bar{b} & iy
	\end{bmatrix}H ; \ x,y \in \R, \ |b|^2 - xy = 1\right\} &
\end{align*}

Then we use Theorem \ref{th:E} to find which of those real structures extend to an effective equivariant real structure on $X$, and we get the following\\
\scalebox{0.7}{$\tabulinesep=0.8mm
\begin{tabu}{|c|c|c|}
		\hline 
		\text{Models} & \sigma=\sigma_s & \sigma=\sigma_c \\ \hline 
		 \text{(a2)}  &  \begin{array}{c}
		 	\left\{\begin{bmatrix}
		 		a & ix \\ iy & \bar{a}
		 	\end{bmatrix}H ; \ x,y \in \R, \  |a|^2 + xy = 1 \right\}
		 	\\ \sqcup \left\{\begin{bmatrix}
		 		a & b \\ c & -\bar{a}
		 	\end{bmatrix}H ; \ b,c \in \R, \  |a|^2 + bc = -1\right\}
		 \end{array} & \begin{array}{c}
		 \left\{\begin{bmatrix}
		 	a & b \\ \bar{b} & d
		 \end{bmatrix}H ; \ a,d \in \R, \  ad - |b|^2 = 1\right\} \\
		 \sqcup \left\{\begin{bmatrix}
		 	ix & b \\  -\bar{b} & iy
		 \end{bmatrix}H ; \ x,y \in \R, \ |b|^2 - xy = 1\right\}
	 \end{array}\\
		\hline
		\text{(c2),(b2 with $m=1$)}    & \left\{\begin{bmatrix}
			a & ix \\ 0 & a^{-1}
		\end{bmatrix}H; \ |a| = 1, x \in \R \right\} & \left\{\begin{bmatrix}
		0 & b \\ -b^{-1} & iy
	\end{bmatrix}H; \ |b| = 1, y \in \R \right\} \\
		\hline
		\text{(d2), (e2) (f2) and (b2 with $m=2$)}  &\left\{\begin{bmatrix}
			a & 0 \\ 0 & a^{-1}
		\end{bmatrix}H; \ |a| = 1 \right\} & \left\{\begin{bmatrix}
			0 & b \\ -b^{-1} & 0
		\end{bmatrix}H; \ |b| = 1 \right\} \\
		\hline	
		\text{(g2)} & \left\{I_2H ; \begin{bmatrix}
			0 & i \\ i & 0
		\end{bmatrix}H; \begin{bmatrix}
		i & i \\ 0 & -i
	\end{bmatrix}H; \begin{bmatrix}
	i & 0 \\ -i & -i
	\end{bmatrix}H\right\} &
	\left\{\begin{bmatrix}
	    0 & 1 \\ -1 & 0
	\end{bmatrix}H ; \begin{bmatrix}
	-i & 0 \\ 0 & i
	\end{bmatrix}H; \begin{bmatrix}
	-i & i \\ i & 0
	\end{bmatrix}H; \begin{bmatrix}
	0 & i \\ -i & i
	\end{bmatrix}H\right\} \\
		\hline
		\begin{array}{c}
			\text{(b2 with $m\geq 3$) and all} \\ \text{the $x_i$ are congruent modulo $\pi$}
		\end{array} &
			 \begin{bmatrix}
				e^{-i\pi x} & 0 \\ 0 & e^{i\pi x}
			\end{bmatrix}H &
		 \begin{bmatrix}
			0 & e^{i\pi x} \\ -e^{-i\pi x} & 0
		\end{bmatrix}H\\
		\hline
		\begin{array}{c}
			\text{(b2 with $m\geq 3$) and not all} \\ \text{the $x_i$ are congruent modulo $\pi$}
		\end{array} &
		\varnothing &
		\varnothing\\
	\hline
\end{tabu}$}\\

Finally we use the automorphism groups computed in Corollary \ref{prop : Aut 2} to compute the equivalence classes of the previous effective equivariant real structures.
\end{proof}

When $H$ is cyclic of order $k \geq 3$, the diagrams of the minimal smooth completions of $\SL_2/H$ have been determined by Bousquet. They depend on whether $k$ is odd or even, we will start by assuming that $k$ is odd. In this case, the diagrams of the minimal smooth completions of $\SL_2/H$, up to conjugacy, are given in Figure 3 of the \hyperref[appendix: A]{Appendix}.

\begin{proposition} \label{prop 4.4}
		Let $k \geq 3$ be an odd integer and $X$ be a minimal smooth completions of $\SL_2/H$. Then there is no effective $(G,\sigma_c)$-equivariant real structures on $X$, and an unique effective $(G,\sigma_s)$-equivariant real structures on $X$ up to equivalence: $\mu_{Id} : gH \mapsto \sigma_s(g)H$.
\end{proposition}

\begin{proof}
    First, we have in this case $N_G(H) = \begin{bmatrix}
        * & 0 \\ 0 & *
    \end{bmatrix}$ and 
\begin{align*}
    Z^1(\Gamma,N_G(H)/H) = & \left\{\begin{bmatrix}
		a & 0 \\ 0 & a^{-1}
	\end{bmatrix}H; \ |a|=1 \right\}  & \text{if $\sigma = \sigma_s$}\\ 
	\sqcup & \left\{\begin{bmatrix}
		0 & ix \\ ix^{-1} & 0
	\end{bmatrix}H; \  x\in \R^* \right\} & \\
    Z^1(\Gamma,N_G(H)/H) = & \left\{\begin{bmatrix}
		x & 0 \\ 0 & x^{-1}
	\end{bmatrix}H ; \ x \in \R^* \right\} &  \text{if $\sigma = \sigma_c$}. \\ & &
\end{align*}

Then by Theorem \ref{th:E}, there are no effective equivariant real structures on $X$ for $\sigma_c$, and the ones for $\sigma_s$ are given by $tH \in \left\{\begin{bmatrix}
		s & 0 \\ 0 & s^{-1}
	\end{bmatrix}H; \ |s|=1 \right\} $. We have $\begin{bmatrix}
	* & 0 \\ 0 & *
\end{bmatrix} \subset  \Aut_{\C}^G(X)$ by Corollary \ref{prop : Aut 2}. Taking $nH = \begin{bmatrix}
x & 0 \\ 0 & x^{-1}
\end{bmatrix} \in \Aut_{\C}^G(X)$ with $x^{-2} =s$ gives the equivalence we need.
\end{proof}

Finally when $k \geq 4$ and is even, the diagrams of the minimal smooth completions of $\SL_2/H$ are given, up to conjugacy, in Figure 4 of the \hyperref[appendix:
A]{Appendix}.

\begin{proposition} \label{prop 4.5}
	Let $k \geq 4$ be an even integer, $\sigma \in \{\sigma_s,\sigma_c\}$ and $X$ be a minimal smooth completions of $\SL_2/A_k$, Then the following table gives, in each case, a representant $\mu$ for each class of effective $(G,\sigma)$-equivariant real structures on $X$.
\end{proposition}

\[\tabulinesep=0.8mm
\begin{tabu}{|c|c|c|}
	\hline 
	\text{Models} & \mu \text{ for } \sigma=\sigma_s & \mu \text{ for } \sigma=\sigma_c \\ \hline 
	\text{(a4),...,(f4),(i4)}  & 
		gH \mapsto \sigma_s(g)H &  gH \mapsto \sigma_c(g)eH\\
	\hline
	\begin{array}{c}
		\text{(g4)} \\ (x = \text{arg}(\alpha))
	\end{array} & \begin{array}{c}
	gH \mapsto \sigma_s(g)\begin{bmatrix}
		e^{-i\pi x} & 0 \\ 0 & e^{i\pi x}
	\end{bmatrix}H
	 \\ gH \mapsto \sigma_s(g)\begin{bmatrix}
  	0 & i|\alpha| \\ i|\alpha|^{-1} & 0
	\end{bmatrix}H \end{array} & \begin{array}{c}
	gH \mapsto \sigma_c(g)\begin{bmatrix}
	0 & e^{i\pi x} \\ -e^{-i\pi x} & 0
	\end{bmatrix}H
	\\ gH \mapsto \sigma_c(g)\begin{bmatrix}
	i|\alpha|^{-1} & 0 \\ 0 & -i|\alpha|
	\end{bmatrix}H \end{array} \\
	\hline
		\begin{array}{c}
		\text{(g4')} \\ (x = \text{arg}(\alpha))
	\end{array}  & gH \mapsto \sigma_s(g)\begin{bmatrix}
	e^{-i\pi x} & 0 \\ 0 & e^{i\pi x}
	\end{bmatrix}H  &  gH \mapsto \sigma_c(g)\begin{bmatrix}
	0 & e^{i\pi x} \\ -e^{-i\pi x} & 0
\end{bmatrix}H\\ 
	\hline
	\text{(h4)}  & \begin{array}{c}
		gH \mapsto \sigma_s(g)H \\ gH \mapsto \sigma_s(g)eH \\ gH \mapsto \sigma_s(g)e\w_{2n}H
	\end{array} & \begin{array}{c}
		gH \mapsto \sigma_c(g)H \\ gH \mapsto \sigma_c(g)eH \\ gH \mapsto \sigma_c(g)\w_{2n}H
	\end{array} \\ 
	\hline
\end{tabu}\]

\begin{proof}
Here we have
\begin{align*}
    Z^1(\Gamma,N_G(H)/H) = & \left\{\begin{bmatrix}
			a & 0 \\ 0 & a^{-1}
	\end{bmatrix}H; \ |a|=1 \right\}  & \text{if $\sigma = \sigma_s$}\\ 
	\sqcup & \left\{\begin{bmatrix}
			0 & ix \\ ix^{-1} & 0
	\end{bmatrix}H; \  x\in \R^* \right\}\\ 
    \sqcup & \left\{\begin{bmatrix}
			0 & ix\zeta_{2k} \\ ix^{-1}\zeta_{2k}^{-1} & 0
	\end{bmatrix}H; \ x \in \R^* \right\} & \\
    Z^1(\Gamma,N_G(H)/H) = & \left\{\begin{bmatrix}
		x & 0 \\ 0 & x^{-1}
	\end{bmatrix}H ; \ x \in \R^* \right\} &  \text{if $\sigma = \sigma_c$} \\ \sqcup & \left\{\begin{bmatrix}
		x\zeta_{2k} & 0 \\ 0 & x^{-1}\zeta_{2k}^{-1}
	\end{bmatrix}H ; \ x \in \R^* \right\} & \\
    \sqcup & \left\{\begin{bmatrix}
		0 & b \\ -b^{-1} & 0
	\end{bmatrix}H ; \ |b| = 1\right\}. &
\end{align*}
Then the following table gives, in each cases, for which $tH \in Z^1(\Gamma,G/H)$ the corresponding equivariant real structure $\mu_t : g \mapsto \sigma(g)tH$ on $X_0 = G/H$ extends to an effective ($G,\sigma$)-equivariant real structure on $X$:	

\scalebox{0.9}{$\tabulinesep=0.8mm
\hspace{-5mm}\begin{tabu}{|c|c|c|}
		\hline 
		\text{Models} & \sigma=\sigma_s & \sigma=\sigma_c \\ \hline 
		\text{(a4),...,(f4),(i4)}  &
			 \left\{\begin{bmatrix}
			a & 0 \\ 0 & a^{-1}
		\end{bmatrix}H ; \  |a| = 1\right\} & 
		\left\{\begin{bmatrix}
			0 & b \\ -b^{-1} & 0
		\end{bmatrix}H ; \ |b| = 1\right\} \\
 		\hline
		\text{(h4)}    & Z^1(\Gamma,G/A_k) & Z^1(\Gamma,G/A_k)\\
		\hline
		\begin{array}{c}
			\text{(g4)} \\ (x = \text{arg}(\alpha))
		\end{array} & \left\{\begin{bmatrix}
			e^{-i\pi x} & 0 \\ 0 & e^{i\pi x}
		\end{bmatrix}; \begin{bmatrix}
		0 & i|\alpha|\\ i|\alpha|^{-1} & 0
	\end{bmatrix} \right\} & 
	 \left\{\begin{bmatrix}
	 	0 & e^{i\pi x} \\ -e^{-i\pi x} & 0
	 \end{bmatrix}; \begin{bmatrix}
	 i|\alpha|^{-1} & 0 \\ 0 & -i|\alpha|
 \end{bmatrix}\right\} \\
		\hline	
				\begin{array}{c}
			\text{(g4')} \\ (x = \text{arg}(\alpha))
		\end{array} &  \begin{bmatrix}
			e^{-i\pi x} & 0 \\ 0 & e^{i\pi x}
		\end{bmatrix} &
	\begin{bmatrix}
		0 & e^{i\pi x} \\ -e^{-i\pi x} & 0
	\end{bmatrix}\\
		\hline
\end{tabu}$}\\

As in the previous case, it is relatively straightforward to find the equivalence classes of those real structures with Corollary \ref{prop : Aut 2}. Note that for model (g4) the two real structures we obtained in both cases are not equivalent in the homogeneous case \cite{MJT20}, so they are not equivalent for $X$.
\end{proof}

\section{Proof of the Main Theorem} \label{section: theorem}

We can now use the previous results to prove the main theorem. In this section, we always consider real structures up to equivalence and real forms up to isomorphism. We say that a real form (in the classical sense) $Y$ of a complex variety is equivariant under an action of $\SL_{2,\R}$ (resp of $\Spin_{3,\R}$) when there is an action of $\SL_{2,\R}$ (resp of $\Spin_{3,\R}$) on $Y$ such that $Y_{\C}$ is an almost homogeneous $\SL_{2,\C}$-threefold.

We recall that the correspondence between the geometric description and the description from the Luna-Vust point of view is given in the \hyperref[appendix: A]{Appendix}.

To begin with, note that all the geometrical descriptions given in Section \ref{section: setting} are valid over any field of characteristic zero. In particular, all the minimal $\SL_2$-threefolds have a trivial real form given by the same description. Furthermore, it is easily checked, by a case by case study, that those real forms are all rational. 

We first prove the following lemma, which helps us determine when a real form has no real points. Recall that a rational variety always has a real point (in fact a dense subset of real points).

\begin{lemma} \label{lemma: points}
    Let $X$ be a smooth almost homogeneous $\SL_2$-variety. If the intersection of $X(\R)$ with the open orbit of $X$ is empty, then $X(\R) = \varnothing$.
\end{lemma}
\begin{proof}
    By \cite[Corollary 2.2.10]{Man20}, the real locus of a smooth real variety is either dense or empty. Hence, either $X(\R)$ intersects the open orbit of $X$ or $X(\R) = \varnothing$. 
\end{proof}

\begin{remark}
    For $H$ a finite subgroup of $\SL_2$ and $\sigma \in \{ \sigma_s,\sigma_c\}$, the list of all $(G,\sigma)$-equivariant real structures on $\SL_2/H$, as well as the corresponding real loci, can be found in \cite[Appendix C]{MJT20}. We can therefore use the previous lemma to determine if $X$ has a real point for all minimal smooth completions $X$ of $\SL_2/H$.
\end{remark}

We now  split the proof of the \hyperref[Main Theorem]{Main Theorem} into a case by case study, starting with the Umemura $\p^1$-bundles $X_k(n,m) \rightarrow \F_k$. 

\begin{proposition}
    The $\p^1$-bundle $X_k(n,m) \rightarrow \F_k$ always has a unique real $\SL_{2,\R}$-form, which is rational. Moreover, it has no real $\Spin_{3,\R}$-form when $k$ is odd, and a unique real $\Spin_{3,\R}$-form, which has no real points, when $k$ is even. Finally, all real forms of $X_k(n,m)$ are equivariant under an action of either $\SL_{2,\R}$ or $\Spin_{3,\R}$.
\end{proposition}
\begin{proof}
We first suppose that $k$ is odd. Then, by Propositions \ref{prop 4.1} and \ref{prop 4.4}, there is only one $(G,\sigma_s)$-equivariant real structure on $X_k(n,m)$ and no $(G,\sigma_c)$-equivariant real structure. Hence, when $k$ is odd, $X_k(n,m)$ has one real $\SL_{2,\R}$-form which must be the trivial one and is therefore rational.

On the other hand when $k$ is even, by Propositions \ref{prop 4.3} and \ref{prop 4.5}, we have one $(G,\sigma)$-equivariant real structure for each $\sigma \in \{\sigma_s, \sigma_c\}$. For $\sigma_c$, the real locus of the open orbit is empty by \cite[Appendix C]{MJT20}. Therefore by Lemma \ref{lemma: points}, when $k$ is even, $X_k(n,m)$ has a rational real $\SL_{2,\R}$-form and a real $\Spin_{3,\R}$-form with no real points.
Finally, there are no other real forms by \cite[Proposition 7.1]{TZ} and its proof.
\end{proof}

\begin{proposition}
    The decomposable $\p^2$-bundle $Z(1-k) \rightarrow \p^1$ has a unique real $\SL_{2,\R}$-form, which is rational, for all $k$. Moreover, it has no real $\Spin_{3,\R}$-form when $k$ is odd, and a unique real $\Spin_{3,\R}$-form, which has no real points, when $k$ is even. Finally, all real forms of $Z(1-k)$ are equivariant under an action of either $\SL_{2,\R}$ or $\Spin_{3,\R}$.
\end{proposition}
\begin{proof}
    The proof is identical to the case of the Umemura $\p^1$-bundles. The non-equivariant real forms are in \cite[Proposition 5.3 and Remark 5.4]{TZ}.
\end{proof}

\begin{proposition}
    For $n \geq 2$, the $\p^1$-bundle $W(n+1) \rightarrow \p^2$ has a unique real form. This real form is rational and is equivariant under an action of $\SL_{2,\R}$, but not under an action of $\Spin_{3,\R}$.
    On the other hand, the variety $W(1)$ has three real forms : \begin{enumerate}[leftmargin = 3mm]
        \item[-] The trivial real form which is rational and is equivariant under an action of $\SL_{2,\R}$ but not of $\Spin_{3,\R}$.
        \item[-] A rational real form which is equivariant under an action of both $\Spin_{3,\R}$ and $\SL_{2,\R}$.
        \item[-] A real form with no real points which is equivariant under an action of $\Spin_{3,\R}$ but not of $\SL_{2,\R}$.
    \end{enumerate}
\end{proposition}
\begin{proof}
For $n \geq 2$, there is, by Proposition \ref{prop 4.1}, only one $(G,\sigma_s)$-equivariant real structure on $W(n+1)$ and no $(G,\sigma_c)$-equivariant real structure. Hence, $W(n+1)$ only has one real $\SL_{2,\R}$-form, which is the trivial one and is in particular rational. By \cite[Corollary 7.2]{TZ}, there are no other real forms for $W(n+1)$.

The variety $W(1)$ can also be defined as the complex flag variety $\PGL_{3,\C}/B$ where $B$ is a Borel subgroup $\PGL_{3,\C}$. By \cite[Proposition 8.1]{TZ}, this variety has three real forms, two rational ones and one with no real points. We only obtained one real structure with an empty real locus which therefore corresponds to the unique real form with no real points. For the other real structures, the $\Gamma$-action given by the real structure defined on the open orbit by $g \mapsto \sigma_s(g)$ acts trivially on the combinatorial data, whereas the other two real structures with a non-empty real locus do not. Therefore the real structure given by $g \mapsto \sigma_s(g)$ corresponds to the trivial real form and the other two real structures correspond to the nontrivial real form.
\end{proof}

Let us now look at the three isolated models, namely $\p^3 \simeq \p(R_1 \oplus R_1)$, $\p(R_2) \times \p(R_1)$ and $Q_3 \subset \p^4 \simeq \p(R_1 \oplus R_1 \oplus R_0)$.
\begin{proposition}
    The variety $\p^3 \simeq \p(R_1 \oplus R_1)$ has one rational real form and one real form with no real points. Those two varieties are both equivariant under an action of $\SL_{2,\R}$ and an action of $\Spin_{3,\R}$.
\end{proposition}

\begin{proof}
Let $\sigma \in \{\sigma_s,\sigma_c\}$, by Proposition \ref{prop 4.1} there are two $(G,\sigma)$-equivariant real structures, with one of them having an empty real locus. However, it is known that $\p^3$ has two real forms, one is rational while the other has no real points. Therefore the real $\SL_{2,\R}$-form with no real points corresponds to the same real form as the real $\Spin_{3,\R}$-form with no real points. The same can be said for the rational real form of $\p^3$.
\end{proof}

\begin{proposition}
    The variety $\p(R_2) \times \p(R_1)$ has two real forms. One is rational and is equivariant under an action of both $\SL_{2,\R}$ and $\Spin_{3,\R}$, while the other one has no real points and is not equivariant under an action of a real form of $\SL_{2,\C}$.
\end{proposition}
\begin{proof}
In this case there is one $(G,\sigma_s)$-equivariant real structure given by $gH \mapsto \sigma_s(g) \begin{bmatrix}
    e^{-i\pi x} & 0 \\ 0 & e^{i\pi x}
\end{bmatrix}H$. On the open orbit, this real structure is equivalent to $gH \mapsto \sigma_s(g)H$ which has a non-empty real locus. Therefore this real structure corresponds to a real form with real points. Similarly, there is one $(G,\sigma_c)$-equivariant real structure which corresponds to a real form with real points. Since $ \p^2 \times \p^1 \simeq \p(R_2) \times \p(R_1)$ has only one real form with real points, which is rational, the previous two equivariant real structures correspond to the same rational real form. Moreover, $ \p^2 \times \p^1$ has a unique real form with no real points, which is therefore not equivariant under an action of a real form of $\SL_{2,\C}$.
\end{proof}

Let $r,s \in \mathbb{N}$, we denote by $Q^{r,s}$ the smooth real quadric hypersurface of equation $\{x_1^2 +...+ x_r^2 - x_{r+1}^2 -...- x_{r+s}^2= 0\} \subset \p_{\R}^{r+s}$.

\begin{proposition}
    The complex quadric $Q_3 \subset \p^4$ has three real forms : \begin{enumerate}[leftmargin = 7mm]
        \item[-] $Q^{5,0}$, which has no real points, and is equivariant both under an action of $\SL_{2,\R}$ and of $\Spin_{3,\R}$.
        \item[-] $Q^{4,1}$, which is rational, and is equivariant under an action of $\Spin_{3,\R}$ but not of $\SL_{2,\R}$.
        \item [-] $Q^{3,2}$, which is rational, and is equivariant under an action of $\SL_{2,\R}$ but not of $\Spin_{3,\R}$.
    \end{enumerate}
\end{proposition}
\begin{proof}
    It is classical that $Q_3$ has three real forms: $Q^{3,2}$ and $Q^{4,1}$, which are both rational, and $Q^{5,0}$, which has no real points (see e.g. \cite[Proposition 10.2]{TZ}). The equivariant real structures on $Q_3$ are given in Proposition \ref{prop 4.3} (model (a2)). The real structures corresponding to $g \mapsto \sigma_c(g)eH$ and to $g \mapsto \sigma_s(g)eH$ have an empty real locus and therefore correspond to the real form $Q^{5,0}$. Moreover, by \cite[Lemma 1.19]{Nak89}, the real form $Q^{4,1}$ is not equivariant under any action of $\SL_{2,\R}$. Hence, the real structure $g \mapsto \sigma_s(g)H$ corresponds to the real form $Q^{3,2}$. Furthermore, the group $\text{PSU}_2$ cannot be embedded into $\SO(3,2)$, which is the automorphism group of $Q^{3,2}$. Therefore, the real structure given by $g \mapsto \sigma_c(g)H$ corresponds to the real form $Q^{4,1}$.
\end{proof}

We now consider the decomposable $\p^1$-bundles $Y(a,b)\rightarrow \p^1 \times \p^1$.

\begin{proposition} Let $a$ and $b$ be two integers such that $a > b, a+b \leq 0$.
    \begin{enumerate}[leftmargin = 7mm]
        \item[-] If $a+b$ is odd, then $Y(a,b)$ has three real forms. One is rational (the trivial one) and equivariant under an action of $\SL_{2,\R}$ but not of $\Spin_{3,\R}$, while the other two have no real points and are not equivariant under an action of a real form of $\SL_{2,\C}$.
        \item[-] If $a+b \neq 0$  and $a+b$ is even, then $Y(a,b)$ has three real forms. The trivial one, which is equivariant under an action of $\SL_{2,\R}$ but not of $\Spin_{3,\R}$, and two without real points: one is equivariant under an action of $\Spin_{3,\R}$ but not of $\SL_{2,\R}$, while the other one is not equivariant under an action of a real form of $\SL_{2,\C}$.
        \item[-] If $a+b = 0$, then $Y(a,b)$ has five real forms. Two that are equivariant under an action of $\SL_{2,\R}$ but not of $\Spin_{3,\R}$, a rational one and one without real points. Two that are equivariant under an action of $\Spin_{3,\R}$ but not of $\SL_{2,\R}$, a rational one and one without real points. And one without real points that is not equivariant under an action of a real form of $\SL_{2,\C}$.
    \end{enumerate}
\end{proposition}

\begin{proof}
    We use \cite[Propositions 5.7 and 5.12, and Remark 5.8 ]{TZ} and its proof to determine all the real forms of $Y(a,b)$ and which ones are rational or without real points. Then, similarly as in the previous cases, we deduce which ones are equivariant under an action of $\SL_{2,\R}$ (respectively of $\Spin_{3,\R}$) from the results of Section \ref{section: forms}.
\end{proof}

It remains to treat the case of the models $Y_0(k)$ with $k \geq 2$ even.  The variety $Y_0(2)$ corresponds to $\p^1 \times \p^1 \times \p^1 \simeq \p(R_1) \times \p(R_1) \times \p(R_1)$. This case was studied in \cite[Example 4.11]{MJT20} and \cite[Lemma 5.9]{TZ}. Recall that $Y_0(k)$ is not minimal.

\begin{proposition}
    Let $k \geq 6$ be an even integer. Then $Y_0(k)$ has one rational $\SL_{2,\R}$-real form, one rational $\Spin_{3,\R}$-real form, one $\SL_{2,\R}$-real form with no real points and one $\Spin_{3,\R}$-real form with no real points. Moreover, $Y_0(k)$ has no other real form.
\end{proposition}
\begin{proof} 
Those varieties appear as $\hat{S_b}$ in \cite{BFT23} with $b = k/2$, and by Lemma 3.5.5 of loc.cit. we have that $\Aut^{\circ}(\hat{S_b}) \simeq \PGL_{2,\C}$. Hence, by \cite[Proposition 2.20]{TZ}, if $X$ is a real form of $\hat{S_b}$, then $\Aut^{\circ}(X)$ is a real form of $\PGL_{2,\C}$, i.e. $\Aut^{\circ}(X) \simeq \PGL_{2,\R}$ or $\Aut^{\circ}(X) \simeq \SO_{3,\R}$. However, there is a nontrivial homomorphism from $\SL_{2,\R}$ to $\PGL_{2,\R}$, the natural $2:1$ cover, but none to $\SO_{3,\R}$. Similarly, there is a nontrivial homomorphism from $\Spin_{3,\R}$ to $\SO_{3,\R}$, the natural $2:1$ cover, but none to $\PGL_{2,\R}$. Therefore, $X$ is equivariant under an action of either $\SL_{2,\R}$ or $\Spin_{3,\R}$.

Recall by Proposition $\ref{prop 4.5}$ that there are four equivariant real structures on $Y_0(k)$: $\mu_1 : gH \mapsto \sigma_s(g) \begin{bmatrix}
    e^{-i\pi x} & 0 \\ 0 & e^{i\pi x}
\end{bmatrix}H$, $\mu_2 : gH \mapsto \sigma_s(g) \begin{bmatrix}
    0 & i|\alpha| \\ i|\alpha|^{-1} & 0
\end{bmatrix}H$, $\mu_3: gH \mapsto \sigma_c(g) \begin{bmatrix}
    0 & e^{-i\pi x} \\ -e^{i\pi x} & 0
\end{bmatrix}H$ and $\mu_4: gH \mapsto \sigma_c(g) \begin{bmatrix}
    i|\alpha|^{-1} & 0 \\ 0 & i|\alpha|
\end{bmatrix}H$. By Lemma \ref{lemma: points}, $\mu_2$ and $\mu_4$ correspond to real forms with no real points.

Let us prove that the real form corresponding to the real structure $\mu_1$ is rational. By \cite[Lemma 3.5.5]{BFT23}, we have the following commutative diagram
\[
    \begin{tikzcd}
    & X \ar[dl, swap, "\epsilon"] \ar[dr, "\eta"] &  \\
     Y_0(k) \ar[rr, dashed, "\psi"] \ar[dr] & & Y(\frac{k}{2},-\frac{k}{2}) \ar[dl]\\
     & \p^1 \times \p^1 &
  \end{tikzcd}
    \]
where all maps are $\PGL_2$-equivariant, and the morphisms $\epsilon$ ans $\eta$ are blow-ups of curves. The morphism $\psi^{-1} \circ \mu \circ \psi$ is a birational antiregular involution on $Y_0(k)$. By \cite[Proposition 3.9]{MJT21}, this map is actually an antiregular automorphism because it preserves the skeleton diagram associated to $Y_0(k)$. Therefore, since $\mu$ corresponds to a rational real form of $Y(\frac{k}{2},-\frac{k}{2})$, $\mu$ corresponds to a rational real form of $Y_0(k)$; this real form is the trivial one. Similarly, we prove that the real form corresponding to $\mu_3$ is rational.
\end{proof}
\newpage

\setcounter{secnumdepth}{0}
\section{Appendix: Diagrams of minimal smooth completions of $\SL_2/A_k$}\label{appendix: A}

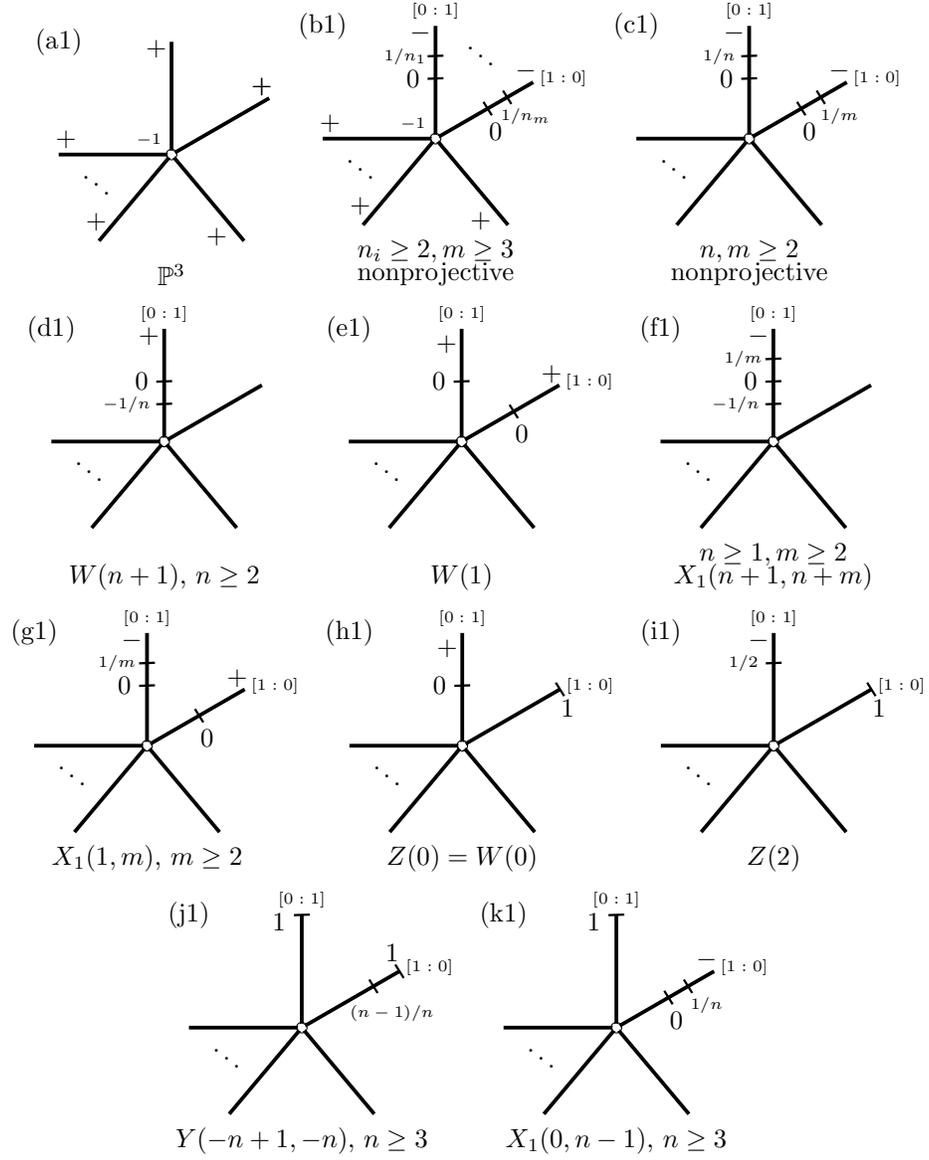
\begin{figure}[ht]
\begin{tikzpicture}[scale=1]
		\path (0,0) coordinate (origin);
		\draw (0,0) circle (2pt) ;
		\path (90:1.5cm) coordinate (P0);
		\path (180:1.5cm) coordinate (P1);
		\path (230:1.5cm) coordinate (P2);
		\path (310:1.5cm) coordinate (P3);
		\path (30:1.5cm) coordinate (P4);
		\draw[line width=0.5mm] (origin) -- (P0) (origin) -- (P1) (origin) -- (P2)(origin) -- (P3) (origin) -- (P4);
		\fill[white] (0,0) circle (1.5pt) ;
		\node at (-1,-.3){$\ddots$};
		\node at (1.2,0.9){$+$};
		\node at (-1.4,0.2){$+$};
		\node at (-1,-.9){$+$};
		\node at (1.2,0.9){$+$};
		\node at (-.2,1.4){$+$};
		\node at (.6,-1.1){$+$};
		\node at (-.3,.2){\tiny{$-1$}};
		\node at (-1.5,1.5){(a1)};
		\node at (0,-1.6){$\p^3$};
\end{tikzpicture}	
\begin{tikzpicture}[scale=1]
		\draw (0,0) circle (2pt) ;
		\path (90:.8cm) coordinate (Q0);
		\path (90:1.1cm) coordinate (Q1);
		\path (30:.8cm) coordinate (Q4);
		\path (30:1.1cm) coordinate (Q5);
		\node at (0,1.7){\tiny{$[0:1]$}};
		\node at (1.7,0.8){\tiny{$[1:0]$}};
		\draw[line width=0.3mm] (Q0) -- +(182:3pt)  (Q0)-- +(2:3pt); ;
		\draw[line width=0.3mm] (Q1) -- +(182:3pt)  (Q1)-- +(2:3pt); ;
		\draw[line width=0.5mm] (origin) -- (P0) (origin) -- (P1) (origin) -- (P2)(origin) -- (P3) (origin) -- (P4);
		\fill[white] (0,0) circle (1.5pt) ;
		\node at (-1,-.3){$\ddots$};
		\node at (.6,1.2){$\ddots$};
		\node at (-.3,.8){$0$};
		\node at (-.2,1.4){$-$};
		\node at (-.4,1.1){\tiny{$1/n_1$}};
		\draw[line width=0.3mm] (Q4) -- +(122:3pt)  (Q4)-- +(302:3pt); ; 
		\draw[line width=0.3mm] (Q5) -- +(122:3pt)  (Q5)-- +(302:3pt); ;
		\node at (1.2,0.9){$-$};
		\node at (1.2,0.3){\tiny{$1/n_m$}};	
		\node at (.8,.1){$0$};
		\node at (-1.4,0.2){$+$};
		\node at (-1,-.9){$+$};
		\node at (.6,-1.1){$+$};
		\node at (-.3,.2){\tiny{$-1$}};
		\node at (-1.5,1.5){(b1)};
		\node at (0,-1.5){$n_i \geq 2, m \geq 3$};
        \node at (0,-1.8){nonprojective};
\end{tikzpicture}
\begin{tikzpicture}[scale=1]
	\draw (0,0) circle (2pt) ;
	\node at (0,1.7){\tiny{$[0:1]$}};
	\node at (1.7,0.8){\tiny{$[1:0]$}};
	\draw[line width=0.3mm] (Q0) -- +(182:3pt)  (Q0)-- +(2:3pt); ;
	\draw[line width=0.3mm] (Q1) -- +(182:3pt)  (Q1)-- +(2:3pt); ;
	\draw[line width=0.5mm] (origin) -- (P0) (origin) -- (P1) (origin) -- (P2)(origin) -- (P3) (origin) -- (P4);
	\fill[white] (0,0) circle (1.5pt) ;
	\node at (-1,-.3){$\ddots$};
	\node at (-.3,.8){$0$};
	\node at (-.2,1.4){$-$};
	\node at (-.4,1.1){\tiny{$1/n$}};
	\draw[line width=0.3mm] (Q4) -- +(122:3pt)  (Q4)-- +(302:3pt); ; 
	\draw[line width=0.3mm] (Q5) -- +(122:3pt)  (Q5)-- +(302:3pt); ;
	\node at (1.2,0.9){$-$};
	\node at (1.2,0.3){\tiny{$1/m$}};	
	\node at (.8,.1){$0$};
	\node at (-1.5,1.5){(c1)};
	\node at (0,-1.5){$n,m \geq 2$};
    \node at (0,-1.8){nonprojective};
 
\end{tikzpicture}

\begin{tikzpicture}[scale=1]
	\draw (0,0) circle (2pt) ;
	\node at (0,1.7){\tiny{$[0:1]$}};
	\node at (1.8,0.8){};
	\path (90:.5cm) coordinate (Q3);
	\draw[line width=0.3mm] (Q0) -- +(182:3pt)  (Q0)-- +(2:3pt); ;
	\draw[line width=0.3mm] (Q3) -- +(182:3pt)  (Q3)-- +(2:3pt); ;
	\draw[line width=0.5mm] (origin) -- (P0) (origin) -- (P1) (origin) -- (P2)(origin) -- (P3) (origin) -- (P4);
	\fill[white] (0,0) circle (1.5pt) ;
	\node at (-1,-.3){$\ddots$};
	\node at (-.3,.8){$0$};
	\node at (-.2,1.4){$+$};
	\node at (-.5,0.5){\tiny{$-1/n$}};
	\node at (-1.5,1.5){(d1)};
	\node at (0,-1.8){$W(n+1)$, $n \geq 2$};
\end{tikzpicture}
\begin{tikzpicture}[scale=1]
	\draw (0,0) circle (2pt) ;
	\node at (0,1.7){\tiny{$[0:1]$}};
	\node at (1.7,0.8){\tiny{$[1:0]$}};
	\draw[line width=0.3mm] (Q0) -- +(182:3pt)  (Q0)-- +(2:3pt); ;
	\draw[line width=0.5mm] (origin) -- (P0) (origin) -- (P1) (origin) -- (P2)(origin) -- (P3) (origin) -- (P4);
	\fill[white] (0,0) circle (1.5pt) ;
	\node at (-1,-.3){$\ddots$};
	\node at (-.3,.8){$0$};
	\node at (-.2,1.3){$+$};
	\draw[line width=0.3mm] (Q4) -- +(122:3pt)  (Q4)-- +(302:3pt); ; 
	\node at (.8,.1){$0$};
	\node at (1.2,0.9){$+$};
	\node at (-1.5,1.5){(e1)};
	\node at (0,-1.8){$W(1)$};
\end{tikzpicture}
\begin{tikzpicture}[scale=1]
	\draw (0,0) circle (2pt) ;
	\node at (0,1.7){\tiny{$[0:1]$}};
	\node at (1.8,0.8){};
	\path (90:.5cm) coordinate (Q3);
	\draw[line width=0.3mm] (Q0) -- +(182:3pt)  (Q0)-- +(2:3pt); ;
	\draw[line width=0.3mm] (Q3) -- +(182:3pt)  (Q3)-- +(2:3pt); ;
	\draw[line width=0.5mm] (origin) -- (P0) (origin) -- (P1) (origin) -- (P2)(origin) -- (P3) (origin) -- (P4);
	\fill[white] (0,0) circle (1.5pt) ;
	\draw[line width=0.3mm] (Q1) -- +(182:3pt)  (Q1)-- +(2:3pt); ;
	\node at (-1,-.3){$\ddots$};
	\node at (-.3,.8){$0$};
	\node at (-.2,1.4){$-$};
	\node at (-.5,0.5){\tiny{$-1/n$}};
	\node at (-.4,1.1){\tiny{$1/m$}};
	\node at (-1.5,1.5){(f1)};
	\node at (0,-1.5){$n \geq 1,m \geq 2$};
    \node at (0,-1.8){$X_1(n+1,n+m)$};
\end{tikzpicture}
\begin{tikzpicture}[scale=1]
		\draw (0,0) circle (2pt) ;
		\node at (0,1.7){\tiny{$[0:1]$}};
		\node at (1.7,0.8){\tiny{$[1:0]$}};
		\draw[line width=0.3mm] (Q0) -- +(182:3pt)  (Q0)-- +(2:3pt); ;
		\draw[line width=0.3mm] (Q1) -- +(182:3pt)  (Q1)-- +(2:3pt); ;
		\draw[line width=0.5mm] (origin) -- (P0) (origin) -- (P1) (origin) -- (P2)(origin) -- (P3) (origin) -- (P4);
		\fill[white] (0,0) circle (1.5pt) ;
		\node at (-1,-.3){$\ddots$};
		\node at (-.3,.8){$0$};
		\node at (-.2,1.4){$-$};
		\node at (-.4,1.1){\tiny{$1/m$}};
		\draw[line width=0.3mm] (Q4) -- +(122:3pt)  (Q4)-- +(302:3pt); ; 
		\node at (.8,.1){$0$};
		\node at (1.2,0.9){$+$};
		\node at (-1.5,1.5){(g1)};
		\node at (0,-1.5){$X_1(1,m)$, $m \geq 2$};
\end{tikzpicture}
	\begin{tikzpicture}[scale=1]
		\draw (0,0) circle (2pt) ;
		\node at (0,1.7){\tiny{$[0:1]$}};
		\node at (1.7,0.8){\tiny{$[1:0]$}};
		\draw[line width=0.3mm] (Q0) -- +(182:3pt)  (Q0)-- +(2:3pt); ;
		\draw[line width=0.5mm] (origin) -- (P0) (origin) -- (P1) (origin) -- (P2)(origin) -- (P3) (origin) -- (P4);
		\fill[white] (0,0) circle (1.5pt) ;
		\node at (-1,-.3){$\ddots$};
		\node at (-.3,.8){$0$};
		\node at (-.2,1.3){$+$};
		\draw[line width=0.3mm] (P4) -- +(122:3pt)  (P4)-- +(302:3pt); 
		\node at (1.4,.5){$1$};
		\node at (-1.5,1.5){(h1)};
		\node at (0,-1.5){$Z(0)=W(0)$};
	\end{tikzpicture}
	\begin{tikzpicture}[scale=1]
			\draw (0,0) circle (2pt) ;
			\node at (0,1.7){\tiny{$[0:1]$}};
			\node at (1.7,0.8){\tiny{$[1:0]$}};
			\draw[line width=0.3mm] (Q1) -- +(182:3pt)  (Q1)-- +(2:3pt); ;
			\draw[line width=0.5mm] (origin) -- (P0) (origin) -- (P1) (origin) -- (P2)(origin) -- (P3) (origin) -- (P4);
			\fill[white] (0,0) circle (1.5pt) ;
			\node at (-1,-.3){$\ddots$};
			\node at (-.4,1.1){\tiny{$1/2$}};
			\node at (-.2,1.4){$-$};
			\draw[line width=0.3mm] (P4) -- +(122:3pt)  (P4)-- +(302:3pt); 
			\node at (1.4,.5){$1$};
			\node at (-1.5,1.5){(i1)};
			\node at (0,-1.5){$Z(2)$};
\end{tikzpicture}
\begin{tikzpicture}[scale=1]
	\draw (0,0) circle (2pt) ;
	\node at (0,1.7){\tiny{$[0:1]$}};
	\node at (1.7,0.8){\tiny{$[1:0]$}};
	\draw[line width=0.3mm] (P0) -- +(182:3pt)  (P0)-- +(2:3pt); ;
	\draw[line width=0.5mm] (origin) -- (P0) (origin) -- (P1) (origin) -- (P2)(origin) -- (P3) (origin) -- (P4);
	\fill[white] (0,0) circle (1.5pt) ;
	\node at (-1,-.3){$\ddots$};
	\node at (-.3,1.4){$1$};
	\draw[line width=0.3mm] (Q5) -- +(122:3pt)  (Q5)-- +(302:3pt);
	\draw[line width=0.3mm] (P4) -- +(122:3pt)  (P4)-- +(302:3pt); 
	\node at (1.2,1){$1$};
	\node at (1.2,0.2){\tiny{$(n-1)/n$}};	
	\node at (-1.5,1.5){(j1)};
	\node at (0,-1.5){$Y(-n+1,-n)$, $n \geq 3$};
\end{tikzpicture}
\begin{tikzpicture}[scale=1]
	\draw (0,0) circle (2pt) ;
	\node at (0,1.7){\tiny{$[0:1]$}};
	\node at (1.7,0.8){\tiny{$[1:0]$}};
	\draw[line width=0.3mm] (P0) -- +(182:3pt)  (P0)-- +(2:3pt); ;
	\draw[line width=0.5mm] (origin) -- (P0) (origin) -- (P1) (origin) -- (P2)(origin) -- (P3) (origin) -- (P4);
	\fill[white] (0,0) circle (1.5pt) ;
	\node at (-1,-.3){$\ddots$};
	\node at (-.3,1.4){$1$};
	\draw[line width=0.3mm] (Q4) -- +(122:3pt)  (Q4)-- +(302:3pt); 
	\draw[line width=0.3mm] (Q5) -- +(122:3pt)  (Q5)-- +(302:3pt); ;
	\node at (1.2,0.9){$-$};
	\node at (1.2,0.3){\tiny{$1/n$}};	
	\node at (.8,.1){$0$};
	\node at (-1.5,1.5){(k1)};
	\node at (0,-1.5){$X_1(0,n-1)$, $n \geq 3$};
\end{tikzpicture}
\caption{List of diagrams of the minimal smooth completions of $\SL_2$.}
\end{figure}

\newpage

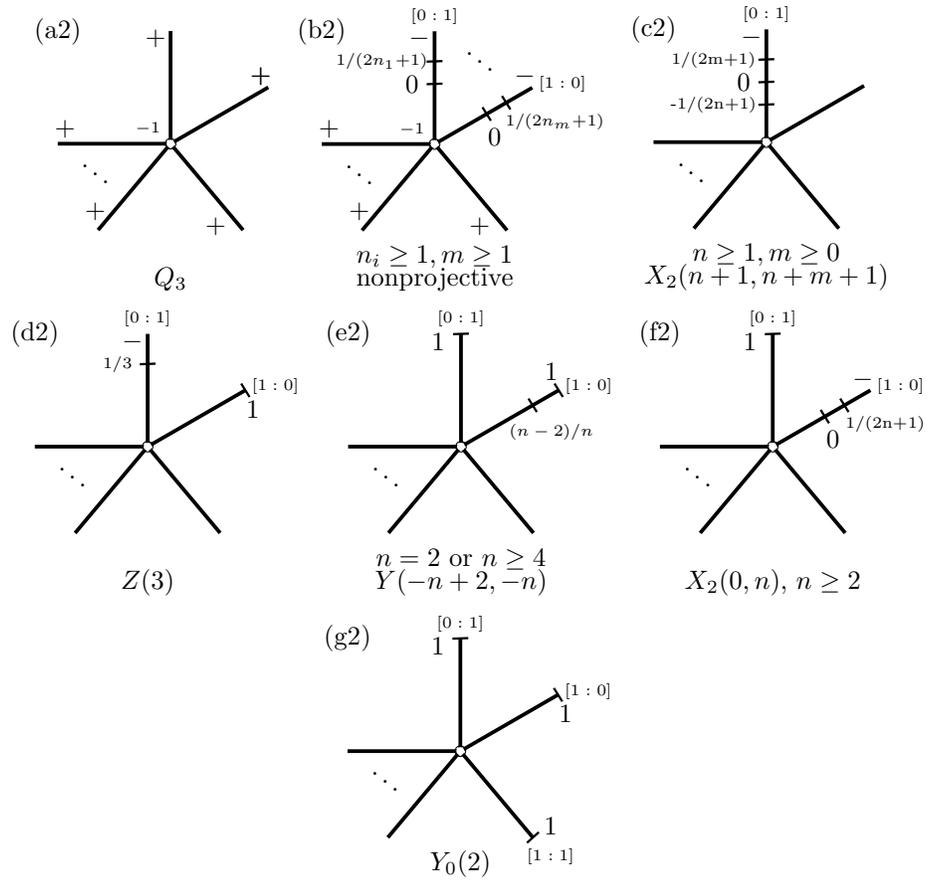
\begin{figure}[ht]
	\begin{tikzpicture}[scale=1]
		\path (0,0) coordinate (origin);
		\draw (0,0) circle (2pt) ;
		\path (90:1.5cm) coordinate (P0);
		\path (180:1.5cm) coordinate (P1);
		\path (230:1.5cm) coordinate (P2);
		\path (310:1.5cm) coordinate (P3);
		\path (30:1.5cm) coordinate (P4);
		\draw[line width=0.5mm] (origin) -- (P0) (origin) -- (P1) (origin) -- (P2)(origin) -- (P3) (origin) -- (P4);
		\fill[white] (0,0) circle (1.5pt) ;
		\node at (-1,-.3){$\ddots$};
		\node at (1.2,0.9){$+$};
		\node at (-1.4,0.2){$+$};
		\node at (-1,-.9){$+$};
		\node at (1.2,0.9){$+$};
		\node at (-.2,1.4){$+$};
		\node at (.6,-1.1){$+$};
		\node at (-.3,.2){\tiny{$-1$}};
		\node at (-1.5,1.5){(a2)};
		\node at (0,-1.8){$Q_3$};
	\end{tikzpicture}
	\begin{tikzpicture}[scale=1]
		\draw (0,0) circle (2pt) ;
		\path (90:.8cm) coordinate (Q0);
		\path (90:1.1cm) coordinate (Q1);
		\path (30:.8cm) coordinate (Q4);
		\path (30:1.1cm) coordinate (Q5);
		\node at (0,1.7){\tiny{$[0:1]$}};
		\node at (1.7,0.8){\tiny{$[1:0]$}};
		\draw[line width=0.3mm] (Q0) -- +(182:3pt)  (Q0)-- +(2:3pt); ;
		\draw[line width=0.3mm] (Q1) -- +(182:3pt)  (Q1)-- +(2:3pt); ;
		\draw[line width=0.5mm] (origin) -- (P0) (origin) -- (P1) (origin) -- (P2)(origin) -- (P3) (origin) -- (P4);
		\fill[white] (0,0) circle (1.5pt) ;
		\node at (-1,-.3){$\ddots$};
		\node at (.6,1.2){$\ddots$};
		\node at (-.3,.8){$0$};
		\node at (-.2,1.4){$-$};
		\node at (-.7,1.1){\tiny{1/(2$n_1$+1)}};
		\draw[line width=0.3mm] (Q4) -- +(122:3pt)  (Q4)-- +(302:3pt); ; 
		\draw[line width=0.3mm] (Q5) -- +(122:3pt)  (Q5)-- +(302:3pt); ;
		\node at (1.2,0.9){$-$};
		\node at (1.6,0.3){\tiny{1/(2$n_m$+1)}};	
		\node at (.8,.1){$0$};
		\node at (-1.4,0.2){$+$};
		\node at (-1,-.9){$+$};
		\node at (.6,-1.1){$+$};
		\node at (-.3,.2){\tiny{$-1$}};
		\node at (-1.5,1.5){(b2)};
		\node at (0,-1.5){$n_i \geq 1, m \geq 1$};
        \node at (0,-1.8){nonprojective};
	\end{tikzpicture}
	\begin{tikzpicture}[scale=1]
		\draw (0,0) circle (2pt) ;
		\node at (0,1.7){\tiny{$[0:1]$}};
		\node at (1.8,0.8){};
		\path (90:.5cm) coordinate (Q3);
		\draw[line width=0.3mm] (Q0) -- +(182:3pt)  (Q0)-- +(2:3pt); ;
		\draw[line width=0.3mm] (Q3) -- +(182:3pt)  (Q3)-- +(2:3pt); ;
		\draw[line width=0.3mm] (Q1) -- +(182:3pt)  (Q1)-- +(2:3pt); ;
		\draw[line width=0.5mm] (origin) -- (P0) (origin) -- (P1) (origin) -- (P2)(origin) -- (P3) (origin) -- (P4);
		\fill[white] (0,0) circle (1.5pt) ;
		\node at (-.7,1.1){\tiny{1/(2m+1)}};
		\node at (-1,-.3){$\ddots$};
		\node at (-.3,.8){$0$};
		\node at (-.2,1.4){$-$};
		\node at (-.7,0.5){\tiny{-1/(2n+1)}};
		\node at (-1.5,1.5){(c2)};
		\node at (0,-1.5){$n \geq 1, m\geq 0$};
     \node at (0,-1.8){$X_2(n+1,n+m+1)$};
	\end{tikzpicture}\\		
	\begin{tikzpicture}[scale=1]
		\draw (0,0) circle (2pt) ;
		\node at (0,1.7){\tiny{$[0:1]$}};
		\node at (1.7,0.8){\tiny{$[1:0]$}};
		\draw[line width=0.3mm] (Q1) -- +(182:3pt)  (Q1)-- +(2:3pt); ;
		\draw[line width=0.5mm] (origin) -- (P0) (origin) -- (P1) (origin) -- (P2)(origin) -- (P3) (origin) -- (P4);
		\fill[white] (0,0) circle (1.5pt) ;
		\node at (-1,-.3){$\ddots$};
		\node at (-.4,1.1){\tiny{$1/3$}};
		\node at (-.2,1.4){$-$};
		\draw[line width=0.3mm] (P4) -- +(122:3pt)  (P4)-- +(302:3pt); 
		\node at (1.4,.5){$1$};
		\node at (-1.5,1.5){(d2)};
		\node at (0,-1.8){$Z(3)$};
	\end{tikzpicture}
	\begin{tikzpicture}[scale=1]
		\draw (0,0) circle (2pt) ;
		\node at (0,1.7){\tiny{$[0:1]$}};
		\node at (1.7,0.8){\tiny{$[1:0]$}};
		\draw[line width=0.3mm] (P0) -- +(182:3pt)  (P0)-- +(2:3pt); ;
		\draw[line width=0.5mm] (origin) -- (P0) (origin) -- (P1) (origin) -- (P2)(origin) -- (P3) (origin) -- (P4);
		\fill[white] (0,0) circle (1.5pt) ;
		\node at (-1,-.3){$\ddots$};
		\node at (-.3,1.4){$1$};
		\draw[line width=0.3mm] (Q5) -- +(122:3pt)  (Q5)-- +(302:3pt);
		\draw[line width=0.3mm] (P4) -- +(122:3pt)  (P4)-- +(302:3pt); 
		\node at (1.2,1){$1$};
		\node at (1.2,0.2){\tiny{$(n-2)/n$}};
		\node at (-1.5,1.5){(e2)};
		\node at (0,-1.5){$n=2$ or $n\geq 4$};
    \node at (0,-1.8){$Y(-n+2,-n)$};
	\end{tikzpicture}
	\begin{tikzpicture}[scale=1]
		\draw (0,0) circle (2pt) ;
		\node at (0,1.7){\tiny{$[0:1]$}};
		\node at (1.7,0.8){\tiny{$[1:0]$}};
		\draw[line width=0.3mm] (P0) -- +(182:3pt)  (P0)-- +(2:3pt); ;
		\draw[line width=0.5mm] (origin) -- (P0) (origin) -- (P1) (origin) -- (P2)(origin) -- (P3) (origin) -- (P4);
		\fill[white] (0,0) circle (1.5pt) ;
		\node at (-1,-.3){$\ddots$};
		\node at (-.3,1.4){$1$};
		\draw[line width=0.3mm] (Q4) -- +(122:3pt)  (Q4)-- +(302:3pt); 
		\draw[line width=0.3mm] (Q5) -- +(122:3pt)  (Q5)-- +(302:3pt); ;
		\node at (1.2,0.9){$-$};
		\node at (1.5,0.3){\tiny{1/(2n+1)}};	
		\node at (.8,.1){$0$};
		\node at (-1.5,1.5){(f2)};
		\node at (0,-1.8){$X_2(0,n)$, $n \geq 2$};
	\end{tikzpicture}\\
	\begin{tikzpicture}[scale=1]
		\draw (0,0) circle (2pt) ;
		\node at (0,1.7){\tiny{$[0:1]$}};
		\node at (1.7,0.8){\tiny{$[1:0]$}};
		\node at (-1,-.3){$\ddots$};
		\draw[line width=0.5mm] (origin) -- (P0) (origin) -- (P1) (origin) -- (P2)(origin) -- (P3) (origin) -- (P4);
		\fill[white] (0,0) circle (1.5pt) ;
		\draw[line width=0.3mm] (P0) -- +(2:3pt)  (P0)-- +(182:3pt); ;
		\draw[line width=0.3mm] (P4) -- +(122:3pt)  (P4)-- +(302:3pt);
		\draw[line width=0.3mm] (P3) -- +(222:3pt)  (P3)-- +(42:3pt); ;
		\node at (1.4,.5){$1$};;
		\node at (-.3,1.4){$1$};
		\node at (1.2,-1){$1$};
		\node at (-1.5,1.5){(g2)};
		\node at (1.2,-1.4){\tiny{$[1:1]$}};
    \node at (0,-1.5){$Y_0(2)$};
	\end{tikzpicture}
	\caption{List of diagrams of the minimal smooth completions of $\PGL_2$.}
\end{figure}

\newpage 

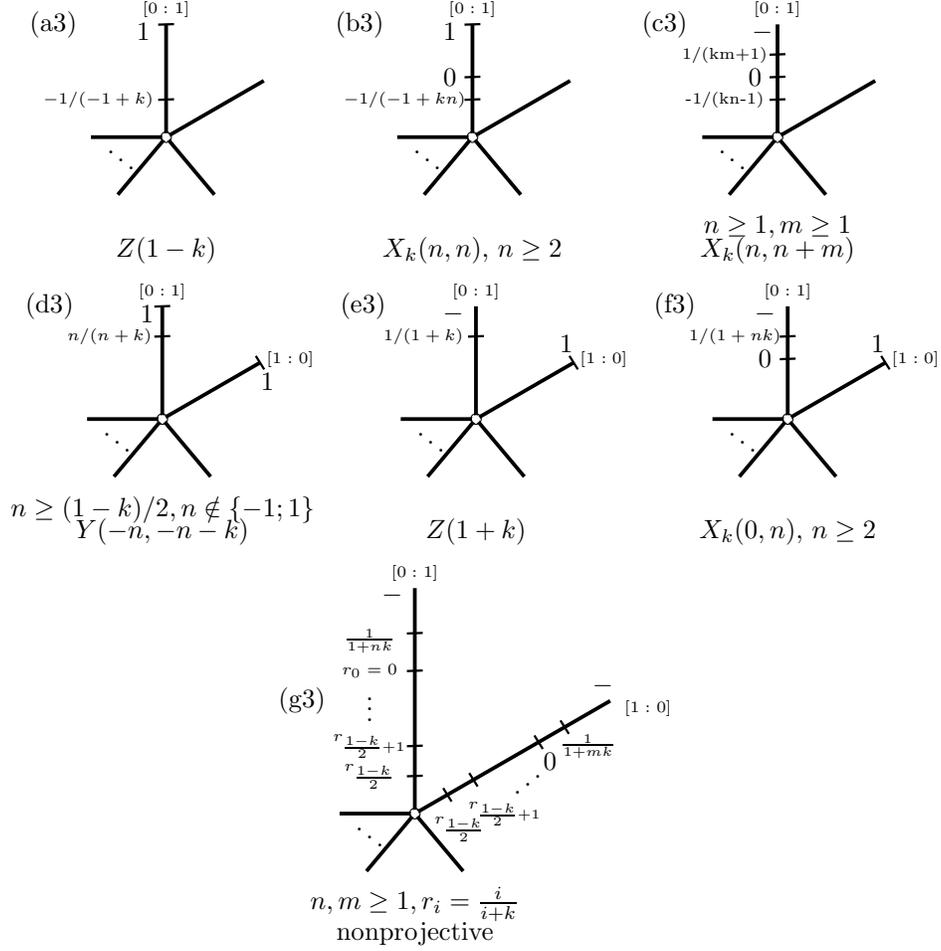
\begin{figure}[ht] \label{figure: odd}
	\begin{tikzpicture}[scale=1]
	\draw (0,0) circle (2pt) ;
	\node at (0,1.7){\tiny{$[0:1]$}};
	\node at (1.9,0.8){};
	\path (90:1.5cm) coordinate (P0);
	\path (30:1.5cm) coordinate (P4);
	\path (230:1cm) coordinate (P2);
	\path (310:1cm) coordinate (P3);
	\path (180:1cm) coordinate (P1);
	\draw[line width=0.3mm] (Q3) -- +(182:3pt)  (Q3)-- +(2:3pt); ;
	\draw[line width=0.3mm] (P0) -- +(182:3pt)  (P0)-- +(2:3pt); ;
	\draw[line width=0.5mm] (origin) -- (P0) (origin) -- (P1) (origin) -- (P2)(origin) -- (P3) (origin) -- (P4);
	\fill[white] (0,0) circle (1.5pt) ;
	\node at (-.6,-.2){$\ddots$};
	\node at (-.9,.5){\tiny{$-1/(-1+k)$}};
	\node at (-.3,1.4){$1$};
	\node at (-1.5,1.5){(a3)};
	\node at (0,-1.5){$Z(1-k)$};
	\end{tikzpicture}
	\begin{tikzpicture}[scale=1]
	\draw (0,0) circle (2pt) ;
	\node at (0,1.7){\tiny{$[0:1]$}};
	\node at (1.9,0.8){};
	\draw[line width=0.3mm] (Q0) -- +(182:3pt)  (Q0)-- +(2:3pt); ;
	\draw[line width=0.3mm] (P0) -- +(182:3pt)  (P0)-- +(2:3pt); ;
	\draw[line width=0.3mm] (Q3) -- +(182:3pt)  (Q3)-- +(2:3pt); ;
	\draw[line width=0.5mm] (origin) -- (P0) (origin) -- (P1) (origin) -- (P2)(origin) -- (P3) (origin) -- (P4);
	\fill[white] (0,0) circle (1.5pt) ;
	\node at (-.6,-.2){$\ddots$};
	\node at (-.3,.8){0};
	\node at (-.9,.5){\tiny{$-1/(-1+kn)$}};
	\node at (-.3,1.4){$1$};
	\node at (-1.5,1.5){(b3)};
	\node at (0,-1.5){$X_k(n,n)$, $n \geq 2 $};
	\end{tikzpicture}
	\begin{tikzpicture}[scale=1]
		\draw (0,0) circle (2pt) ;
		\node at (0,1.7){\tiny{$[0:1]$}};
		\node at (1.9,0.8){};
		\node at (1.8,0.8){};
		\path (90:.5cm) coordinate (Q3);
		\draw[line width=0.3mm] (Q0) -- +(182:3pt)  (Q0)-- +(2:3pt); ;
		\draw[line width=0.3mm] (Q3) -- +(182:3pt)  (Q3)-- +(2:3pt); ;
		\draw[line width=0.3mm] (Q1) -- +(182:3pt)  (Q1)-- +(2:3pt); ;
		\draw[line width=0.5mm] (origin) -- (P0) (origin) -- (P1) (origin) -- (P2)(origin) -- (P3) (origin) -- (P4);
		\fill[white] (0,0) circle (1.5pt) ;
		\node at (-.7,1.1){\tiny{1/(km+1)}};
		\node at (-.6,-.2){$\ddots$};
		\node at (-.3,.8){$0$};
		\node at (-.2,1.4){$-$};
		\node at (-.7,0.5){\tiny{-1/(kn-1)}};
		\node at (-1.5,1.5){(c3)};
		\node at (0,-1.2){$n \geq 1, m\geq 1$};
    \node at (0,-1.5){$X_k(n,n+m)$};
	\end{tikzpicture}\\		
	\begin{tikzpicture}[scale=1]
		\draw (0,0) circle (2pt) ;
		\node at (0,1.7){\tiny{$[0:1]$}};
		\node at (1.7,0.8){\tiny{$[1:0]$}};
		\draw[line width=0.3mm] (P0) -- +(182:3pt)  (P0)-- +(2:3pt); ;
		\draw[line width=0.3mm] (Q1) -- +(182:3pt)  (Q1)-- +(2:3pt); ;
		\draw[line width=0.5mm] (origin) -- (P0) (origin) -- (P1) (origin) -- (P2)(origin) -- (P3) (origin) -- (P4);
		\fill[white] (0,0) circle (1.5pt) ;
		\node at (-.6,-.2){$\ddots$};
		\node at (-.7,1.1){\tiny{$n/(n+k)$}};
		\node at (-.2,1.4){$1$};
		\draw[line width=0.3mm] (P4) -- +(122:3pt)  (P4)-- +(302:3pt); 
		\node at (1.4,.5){$1$};
		\node at (-1.5,1.5){(d3)};
		\node at (0,-1.5){$Y(-n,-n-k)$};
		\node at (0,-1.2){$n \geq (1-k)/2, n \notin \{-1;1\}$};
	\end{tikzpicture}
	\begin{tikzpicture}[scale=1]
		\draw (0,0) circle (2pt) ;
		\node at (0,1.7){\tiny{$[0:1]$}};
		\node at (1.7,0.8){\tiny{$[1:0]$}};
		\node at (-.7,1.1){\tiny{$1/(1+k)$}};
		\draw[line width=0.3mm] (Q1) -- +(182:3pt)  (Q1)-- +(2:3pt); ;
		\draw[line width=0.5mm] (origin) -- (P0) (origin) -- (P1) (origin) -- (P2)(origin) -- (P3) (origin) -- (P4);
		\fill[white] (0,0) circle (1.5pt) ;
		\node at (-.6,-.2){$\ddots$};
		\node at (-.3,1.4){$-$};
		\draw[line width=0.3mm] (P4) -- +(122:3pt)  (P4)-- +(302:3pt); 
		\node at (1.2,1){$1$};
		\node at (-1.5,1.5){(e3)};
		\node at (0,-1.5){$Z(1+k)$};
	\end{tikzpicture}
	\begin{tikzpicture}[scale=1]
	\draw (0,0) circle (2pt) ;
	\node at (0,1.7){\tiny{$[0:1]$}};
	\node at (1.7,0.8){\tiny{$[1:0]$}};
	\node at (-.7,1.1){\tiny{$1/(1+nk)$}};
	\draw[line width=0.3mm] (Q1) -- +(182:3pt)  (Q1)-- +(2:3pt); ;
	\draw[line width=0.3mm] (Q0) -- +(182:3pt)  (Q0)-- +(2:3pt); ;
	\draw[line width=0.5mm] (origin) -- (P0) (origin) -- (P1) (origin) -- (P2)(origin) -- (P3) (origin) -- (P4);
	\fill[white] (0,0) circle (1.5pt) ;
	\node at (-.6,-.2){$\ddots$};
	\node at (-.3,1.4){$-$};
	\draw[line width=0.3mm] (P4) -- +(122:3pt)  (P4)-- +(302:3pt); 
	\node at (1.2,1){$1$};
	\node at (-.3,.8){$0$};
	\node at (-1.5,1.5){(f3)};
	\node at (0,-1.5){$X_k(0,n)$, $n \geq 2$};
	\end{tikzpicture}
	\begin{tikzpicture}[scale=1]
		\draw (0,0) circle (2pt) ;
		\path (90:3cm) coordinate (P0);
		\path (30:3cm) coordinate (P4);
		\path (30:.5cm) coordinate (Q4);
		\path (30:.9cm) coordinate (Q3);
		\path (30:1.9cm) coordinate (Q2);
		\path (30:2.3cm) coordinate (Q1);
		\path (90:.5cm) coordinate (R4);
		\path (90:.9cm) coordinate (R3);
		\path (90:1.9cm) coordinate (R2);
		\path (90:2.4cm) coordinate (R1);
		\node at (0,3.2){\tiny{$[0:1]$}};
		\node at (-.6,-.2){$\ddots$};
		\node at (3.1,1.4){\tiny{$[1:0]$}};
		\draw[line width=0.5mm] (origin) -- (P0) (origin) -- (P1) (origin) -- (P2)(origin) -- (P3) (origin) -- (P4);
		\fill[white] (0,0) circle (1.5pt) ;
		\draw[line width=0.3mm] (Q4) -- +(122:3pt)  (Q4)-- +(302:3pt); 
		\draw[line width=0.3mm] (Q3) -- +(122:3pt)  (Q3)-- +(302:3pt); 
		\draw[line width=0.3mm] (Q2) -- +(122:3pt)  (Q2)-- +(302:3pt); 
		\draw[line width=0.3mm] (Q1) -- +(122:3pt)  (Q1)-- +(302:3pt);	
		\draw[line width=0.3mm] (R1) -- +(2:3pt)  (R1)-- +(182:3pt);
		\draw[line width=0.3mm] (R2) -- +(2:3pt)  (R2)-- +(182:3pt);
		\draw[line width=0.3mm] (R3) -- +(2:3pt)  (R3)-- +(182:3pt);
		\draw[line width=0.3mm] (R4) -- +(2:3pt)  (R4)-- +(182:3pt);
		\node at (2.5,1.7){$-$};;
		\node at (-.3,2.9){$-$};;
		\node at (.6,-.2){\tiny{$r_{\frac{1-k}{2}}$}};
		\node at (1.2,0){\tiny{$r_{\frac{1-k}{2}+1}$}};
		\node at (1.5,.5){$\reflectbox{$\ddots$}$};
		\node at (1.8,.7){0};
		\node at (2.3,.9){\tiny{$\frac{1}{1+mk}$}};
		\node at (-.6,.5){\tiny{$r_{\frac{1-k}{2}}$}};
		\node at (-.6,.9){\tiny{$r_{\frac{1-k}{2}+1}$}};
		\node at (-.6,1.5){$\vdots$};
		\node at (-.6,1.9){\tiny{$r_0 = 0$}};
		\node at (-.6,2.3){\tiny{$\frac{1}{1+nk}$}};
		\node at (-1.5,1.5){(g3)};
		\node at (0,-1.2){$n,m \geq 1, r_i = \frac{i}{i+k}$};
    \node at (0,-1.6){nonprojective};
	\end{tikzpicture}
	\caption{List of diagrams of the minimal smooth completions of $\SL_2/A_k$, where $k \geq 3$ and $k$ is odd.}
\end{figure}

\newpage 

\begin{figure}[ht]
	\begin{tikzpicture}[scale=1]
		\draw (0,0) circle (2pt) ;
		\path (90:.8cm) coordinate (Q0);
		\path (90:1.1cm) coordinate (Q1);
		\path (90:.5cm) coordinate (Q3);
		\path (30:.8cm) coordinate (Q4);
		\path (30:1.1cm) coordinate (Q5);
		\node at (0,1.7){\tiny{$[0:1]$}};
		\node at (1.9,0.8){};
		\path (90:1.5cm) coordinate (P0);
		\path (30:1.5cm) coordinate (P4);
		\path (230:1cm) coordinate (P2);
		\path (310:1cm) coordinate (P3);
		\path (180:1cm) coordinate (P1);
		\draw[line width=0.3mm] (Q3) -- +(182:3pt)  (Q3)-- +(2:3pt); ;
		\draw[line width=0.3mm] (P0) -- +(182:3pt)  (P0)-- +(2:3pt); ;
		\draw[line width=0.5mm] (origin) -- (P0) (origin) -- (P1) (origin) -- (P2)(origin) -- (P3) (origin) -- (P4);
		\fill[white] (0,0) circle (1.5pt) ;
		\node at (-.6,-.2){$\ddots$};
		\node at (-.9,.5){\tiny{$-1/(-1+k)$}};
		\node at (-.3,1.4){$1$};
		\node at (-1.5,1.5){(a4)};
		\node at (0,-1.5){$Z(1-k)$};
	\end{tikzpicture}
	\begin{tikzpicture}[scale=1]
		\draw (0,0) circle (2pt) ;
		\node at (0,1.7){\tiny{$[0:1]$}};
		\node at (1.9,0.8){};
		\draw[line width=0.3mm] (Q0) -- +(182:3pt)  (Q0)-- +(2:3pt); ;
		\draw[line width=0.3mm] (P0) -- +(182:3pt)  (P0)-- +(2:3pt); ;
		\draw[line width=0.3mm] (Q3) -- +(182:3pt)  (Q3)-- +(2:3pt); ;
		\draw[line width=0.5mm] (origin) -- (P0) (origin) -- (P1) (origin) -- (P2)(origin) -- (P3) (origin) -- (P4);
		\fill[white] (0,0) circle (1.5pt) ;
		\node at (-.6,-.2){$\ddots$};
		\node at (-.3,.8){0};
		\node at (-.9,.5){\tiny{$-1/(-1+kn)$}};
		\node at (-.3,1.4){$1$};
		\node at (-1.5,1.5){(b4)};
		\node at (0,-1.5){$X_k(n,n)$, $n \geq 2 $};
	\end{tikzpicture}
	\begin{tikzpicture}[scale=1]
		\draw (0,0) circle (2pt) ;
		\node at (0,1.7){\tiny{$[0:1]$}};
		\node at (1.9,0.8){};
		\node at (1.8,0.8){};
		\path (90:.5cm) coordinate (Q3);
		\draw[line width=0.3mm] (Q0) -- +(182:3pt)  (Q0)-- +(2:3pt); ;
		\draw[line width=0.3mm] (Q3) -- +(182:3pt)  (Q3)-- +(2:3pt); ;
		\draw[line width=0.3mm] (Q1) -- +(182:3pt)  (Q1)-- +(2:3pt); ;
		\draw[line width=0.5mm] (origin) -- (P0) (origin) -- (P1) (origin) -- (P2)(origin) -- (P3) (origin) -- (P4);
		\fill[white] (0,0) circle (1.5pt) ;
		\node at (-.7,1.1){\tiny{1/(km+1)}};
		\node at (-.6,-.2){$\ddots$};
		\node at (-.3,.8){$0$};
		\node at (-.2,1.4){$-$};
		\node at (-.7,0.5){\tiny{-1/(kn-1)}};
		\node at (-1.5,1.5){(c4)};
		\node at (0,-1.2){$n \geq 1, m\geq 1$};
    \node at (0,-1.5){$X_k(n,n+m)$};
	\end{tikzpicture}\\		
	\begin{tikzpicture}[scale=1]
		\draw (0,0) circle (2pt) ;
		\node at (0,1.7){\tiny{$[0:1]$}};
		\node at (1.7,0.8){\tiny{$[1:0]$}};
		\draw[line width=0.3mm] (P0) -- +(182:3pt)  (P0)-- +(2:3pt); ;
		\draw[line width=0.3mm] (Q1) -- +(182:3pt)  (Q1)-- +(2:3pt); ;
		\draw[line width=0.5mm] (origin) -- (P0) (origin) -- (P1) (origin) -- (P2)(origin) -- (P3) (origin) -- (P4);
		\fill[white] (0,0) circle (1.5pt) ;
		\node at (-.6,-.2){$\ddots$};
		\node at (-.7,1.1){\tiny{$n/(n+k)$}};
		\node at (-.2,1.4){$1$};
		\draw[line width=0.3mm] (P4) -- +(122:3pt)  (P4)-- +(302:3pt); 
		\node at (1.4,.5){$1$};
		\node at (-1.5,1.5){(d4)};
		\node at (0,-1.5){$Y(-n,-n-k)$};
		\node at (0,-1.2){$n \geq (1-k)/2, n \notin \{-1;1\}$};
	\end{tikzpicture}
	\begin{tikzpicture}[scale=1]
		\draw (0,0) circle (2pt) ;
		\node at (0,1.7){\tiny{$[0:1]$}};
		\node at (1.7,0.8){\tiny{$[1:0]$}};
		\node at (-.7,1.1){\tiny{$1/(1+k)$}};
		\draw[line width=0.3mm] (Q1) -- +(182:3pt)  (Q1)-- +(2:3pt); ;
		\draw[line width=0.5mm] (origin) -- (P0) (origin) -- (P1) (origin) -- (P2)(origin) -- (P3) (origin) -- (P4);
		\fill[white] (0,0) circle (1.5pt) ;
		\node at (-.6,-.2){$\ddots$};
		\node at (-.3,1.4){$-$};
		\draw[line width=0.3mm] (P4) -- +(122:3pt)  (P4)-- +(302:3pt); 
		\node at (1.2,1){$1$};
		\node at (-1.5,1.5){(e4)};
		\node at (0,-1.5){$Z(1+k)$};
	\end{tikzpicture}
	\begin{tikzpicture}[scale=1]
		\draw (0,0) circle (2pt) ;
		\node at (0,1.7){\tiny{$[0:1]$}};
		\node at (1.7,0.8){\tiny{$[1:0]$}};
		\node at (-.7,1.1){\tiny{$1/(1+nk)$}};
		\draw[line width=0.3mm] (Q1) -- +(182:3pt)  (Q1)-- +(2:3pt); ;
		\draw[line width=0.3mm] (Q0) -- +(182:3pt)  (Q0)-- +(2:3pt); ;
		\draw[line width=0.5mm] (origin) -- (P0) (origin) -- (P1) (origin) -- (P2)(origin) -- (P3) (origin) -- (P4);
		\fill[white] (0,0) circle (1.5pt) ;
		\node at (-.6,-.2){$\ddots$};
		\node at (-.3,1.4){$-$};
		\draw[line width=0.3mm] (P4) -- +(122:3pt)  (P4)-- +(302:3pt); 
		\node at (1.2,1){$1$};
		\node at (-.3,.8){$0$};
		\node at (-1.5,1.5){(f4)};
		\node at (0,-1.5){$X_k(0,n)$, $n \geq 2$};
	\end{tikzpicture}
	\begin{tikzpicture}[scale=1]
	\draw (0,0) circle (2pt) ;
	\node at (0,1.7){\tiny{$[0:1]$}};
	\node at (1.7,0.8){\tiny{$[1:0]$}};
	\node at (1.4,-0.8){\tiny{$\omega = [1 : \alpha]$}};
	\draw[line width=0.3mm] (P3) -- +(222:3pt)  (P3)-- +(42:3pt); 
	\draw[line width=0.3mm] (P0) -- +(182:3pt)  (P0)-- +(2:3pt); 
	\draw[line width=0.3mm] (P4) -- +(122:3pt)  (P4)-- +(302:3pt); 
	\draw[line width=0.5mm] (origin) -- (P0) (origin) -- (P1) (origin) -- (P2)(origin) -- (P3) (origin) -- (P4);
	\fill[white] (0,0) circle (1.5pt) ;
	\node at (-.6,-.2){$\ddots$};
	\node at (-.3,1.4){$1$};
	\node at (.4,-0.95){\tiny{$\frac{4}{k}-1$}};
	\node at (1.2,1){$1$};
	\node at (-1.5,1.5){(g4)};
	\node at (0,-1.3){When $k \geq 6$};
    \node at (0,-1.65){$Y_0(k)$};
\end{tikzpicture}
	\begin{tikzpicture}[scale=1]
	\draw (0,0) circle (2pt) ;
	\node at (0,1.7){\tiny{$[0:1]$}};
	\node at (1.7,0.8){\tiny{$[1:0]$}};
	\node at (1.4,-0.8){\tiny{$\omega = [1 : \alpha]$}};
	\draw[line width=0.3mm] (P3) -- +(222:3pt)  (P3)-- +(42:3pt); ;
	\draw[line width=0.3mm] (P0) -- +(182:3pt)  (P0)-- +(2:3pt); ;
	\draw[line width=0.5mm] (origin) -- (P0) (origin) -- (P1) (origin) -- (P2)(origin) -- (P3) (origin) -- (P4);
	\fill[white] (0,0) circle (1.5pt) ;
	\node at (-.6,-.2){$\ddots$};
	\node at (-.3,1.4){$1$};
	\node at (.4,-0.95){\tiny{$\frac{4}{k}-1$}}; 
	\node at (-1.5,1.5){(g4')};
	\node at (0,-1.3){only for $k=4$};
    \node at (0,-1.65){$\p(R_2) \times \p(R_1)$};
\end{tikzpicture}
\begin{tikzpicture}[scale=1]
	\draw (0,0) circle (2pt) ;
	\node at (0,1.7){\tiny{$[0:1]$}};
	\node at (1.7,0.8){\tiny{$[1:0]$}};
	\draw[line width=0.3mm] (P0) -- +(182:3pt)  (P0)-- +(2:3pt); ;
	\draw[line width=0.5mm] (origin) -- (P0) (origin) -- (P1) (origin) -- (P2)(origin) -- (P3) (origin) -- (P4);
	\fill[white] (0,0) circle (1.5pt) ;
	\node at (-.6,-.2){$\ddots$};
	\draw[line width=0.3mm] (P4) -- +(122:3pt)  (P4)-- +(302:3pt); 
	\node at (1.2,1){$1$};
	\node at (-.3,1.4){$1$};
	\node at (.5,-0.9){$+$};
	\node at (-1,-.2){$+$};
	\node at (-.4,-.8){$+$};
	\node at (-.3,.2){\tiny{$-1$}};
	\node at (-1.5,1.5){(h4)};
	\node at (0,-1.65){$Y(\frac{k}{2},-\frac{k}{2})$};
\end{tikzpicture}
	\begin{tikzpicture}[scale=1]
		\draw (0,0) circle (2pt) ;
		\path (90:3cm) coordinate (P0);
		\path (30:3cm) coordinate (P4);
		\path (30:.5cm) coordinate (Q4);
		\path (30:.9cm) coordinate (Q3);
		\path (30:1.9cm) coordinate (Q2);
		\path (30:2.3cm) coordinate (Q1);
		\path (90:.5cm) coordinate (R4);
		\path (90:.9cm) coordinate (R3);
		\path (90:1.9cm) coordinate (R2);
		\path (90:2.4cm) coordinate (R1);
		\node at (0,3.2){\tiny{$[0:1]$}};
		\node at (-.6,-.2){$\ddots$};
		\node at (3.1,1.4){\tiny{$[1:0]$}};
		\draw[line width=0.5mm] (origin) -- (P0) (origin) -- (P1) (origin) -- (P2)(origin) -- (P3) (origin) -- (P4);
		\fill[white] (0,0) circle (1.5pt) ;
		\draw[line width=0.3mm] (Q4) -- +(122:3pt)  (Q4)-- +(302:3pt); 
		\draw[line width=0.3mm] (Q3) -- +(122:3pt)  (Q3)-- +(302:3pt); 
		\draw[line width=0.3mm] (Q2) -- +(122:3pt)  (Q2)-- +(302:3pt); 
		\draw[line width=0.3mm] (Q1) -- +(122:3pt)  (Q1)-- +(302:3pt); 
		\draw[line width=0.3mm] (R1) -- +(2:3pt)  (R1)-- +(182:3pt);
		\draw[line width=0.3mm] (R2) -- +(2:3pt)  (R2)-- +(182:3pt);
		\draw[line width=0.3mm] (R3) -- +(2:3pt)  (R3)-- +(182:3pt);
		\draw[line width=0.3mm] (R4) -- +(2:3pt)  (R4)-- +(182:3pt);
		\node at (2.5,1.7){$-$};;
		\node at (-.3,2.9){$-$};;
		\node at (.6,-.2){\tiny{$r_{1-\frac{k}{2}}$}};
		\node at (1.2,0){\tiny{$r_{2-\frac{k}{2}}$}};
		\node at (1.5,.5){$\reflectbox{$\ddots$}$};
		\node at (1.8,.7){0};
		\node at (2.3,.9){\tiny{$\frac{1}{1+mk}$}};
		\node at (-.6,.5){\tiny{$r_{1-\frac{k}{2}}$}};
		\node at (-.6,.9){\tiny{$r_{2-\frac{k}{2}}$}};
		\node at (-.6,1.5){$\vdots$};
		\node at (-.6,1.9){\tiny{$r_0 = 0$}};
		\node at (-.6,2.3){\tiny{$\frac{1}{1+nk}$}};
		\node at (-1.5,1.5){(i4)};
		\node at (0,-1.2){$n,m \geq 1, r_i = \frac{i}{i+k}$};
        \node at (0,-1.5){nonprojective};
	\end{tikzpicture}
	\caption{List of diagrams of the minimal smooth completions of $\SL_2/A_k$, where $k \geq 4$ and $k$ is even.}
\end{figure}
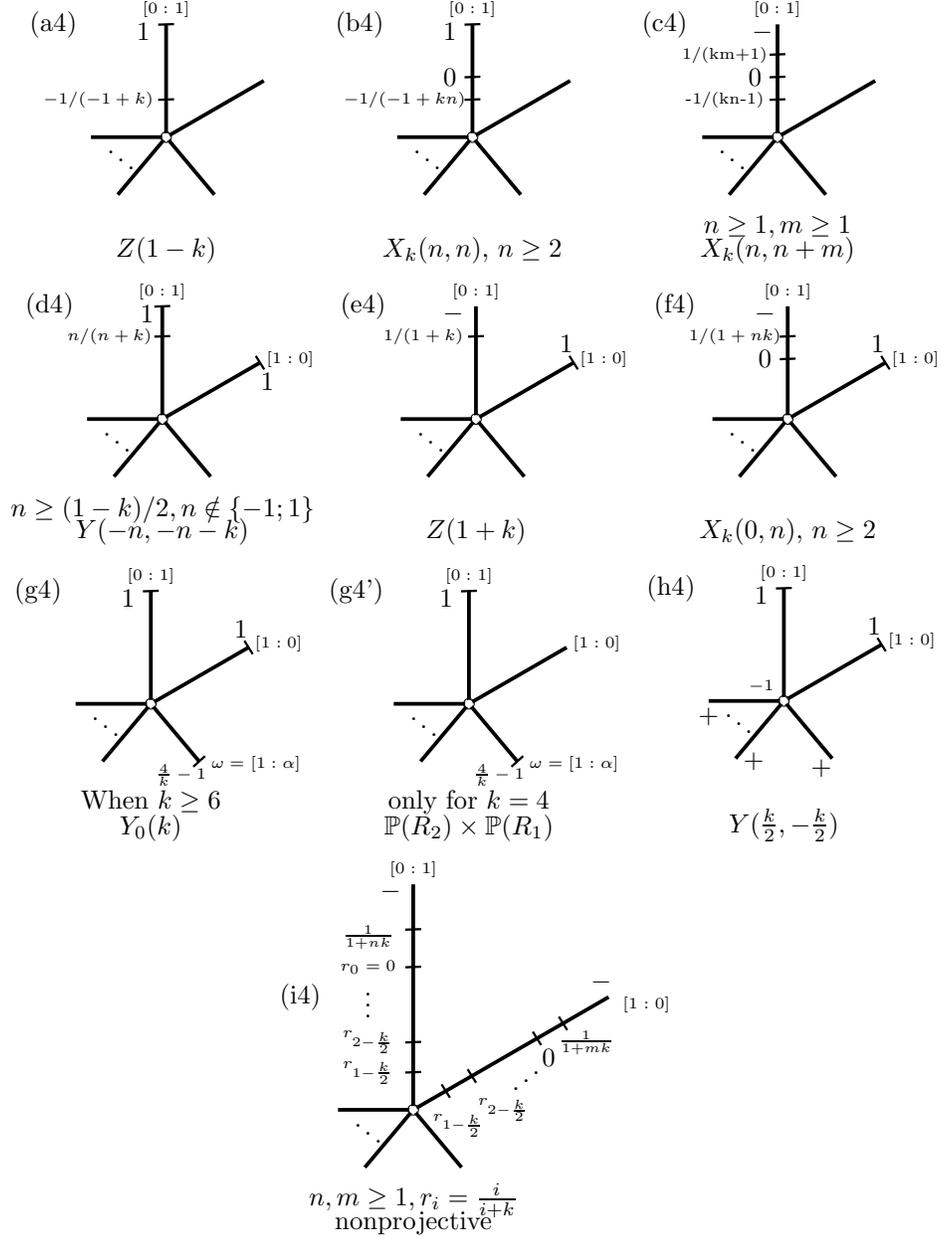

\newpage

\bibliographystyle{alpha}

\end{document}